\documentclass[DIV=12]{scrartcl}

\usepackage{graphicx,amsmath,amssymb,amsfonts,amsthm}

\DeclareMathOperator{\trace}{Trace}
\DeclareMathOperator{\diag}{diag}
\usepackage{algorithmic}
\usepackage{algorithm}
\usepackage{xcolor}

\usepackage[textsize=tiny]{todonotes}

\newtheorem{example}{Example}
\newtheorem{remark}{Remark}
\newtheorem{proposition}{Proposition}

\colorlet{VeryRed}{red!90!black!80}
\colorlet{VeryGreen}{green!90!black!80}

\title{Finite time horizon mixed control of vibrational systems%
\thanks{
This work has been  supported in parts by Croatian Science Foundation under the projects `Control of Dynamical Systems' (IP-2016-06-2468) and `Vibration Reduction in Mechanical Systems' (IP-2019-04-6774).}
}  
\author{Ivica Naki\'{c}\footnote{Faculty od Science, Department of Mathematics, University of Zagreb, Zagreb, Croatia, (ivica.nakic@math.hr)},
        Marinela Pilj Vidakovi\'{c}\footnote{School of Applied Mathematics and Informatics, Osijek, Croatia, (mpilj@mathos.hr, ztomljan@mathos.hr)},
        Zoran Tomljanovi\'{c}\footnotemark[\value{footnote}] }

\begin{document}

\maketitle

\begin{abstract}
We consider a vibrational system control problem over a finite time horizon. The performance measure of the system is taken to be $p$-mixed $H_2$ norm which generalizes the standard $H_2$ norm. We present an algorithm for efficient calculation of this norm in the case when the system is parameter dependent and the number of inputs or outputs of the system is significantly smaller than the order of the system. Our approach is based on a novel procedure which is not based on solving Lyapunov equations and which takes into account the structure of the system. We use a characterization of the $H_2$ norm given in terms of integrals which we solve using adaptive quadrature rules. This enables us to use recycling strategies as well as parallelization. The efficiency of the new algorithm allows for an analysis of  the influence of various system parameters and  different finite time horizons on the value of the $p$-mixed $H_2$ norm. We illustrate our approach by numerical examples concerning an $n$-mass oscillator with one damper.
\end{abstract}


\vspace{0.5cm}

\textbf{  AMS subject classiﬁcations: } 93C05, 93C15, 74D99, 70Q05

\begin{section}{Introduction}
\label{Introduction}
The topic of this paper are vibrational systems, a class of systems which models oscillating physical systems. We are interested in those systems where the vibrations are unwanted and where one wants to design a system which reduces or minimizes the effects of a particular type of vibratory disturbance.

More precisely, we deal with a linear vibrational system given by the following matrix algebraic-differential equation:
\begin{equation}\label{MDK}
	G = \left\{
	\begin{aligned}
	& M\ddot{x}+C\dot{x}+Kx=B_2 u, \\
	& x(0) = x_0, \; \dot{x}(0) = \dot{x}_0,\\
	& w = \begin{bmatrix}
	E_1 x \\ E_2 \dot{x}
\end{bmatrix}.
	\end{aligned}
	\right.
\end{equation}
Here the mass matrix $M$ and the stiffness matrix $K$ are
real, symmetric positive definite matrices of order $n\in \mathbb{N}$.
The damping matrix $C$ is real positive definite matrix and it is given as a sum of the internal damping matrix $C_{\mathrm{int}}$ (which is predefined) and the external damping matrix $C_{\mathrm{ext}}$ (which is subject to a design choice), that is
 $C=C_{\mathrm{int}}+C_{\mathrm{ext}}$.  The internal damping $C_{\mathrm{int}}$ is usually taken to
be a small multiple of the so called critical damping, that is, $C_{\mathrm{int}} =
\alpha  C_{\mathrm{crit}}$, for some $\alpha>0$, where  the critical damping
$C_{\mathrm{crit}}$ is (see, e.g., \cite{TRUHVES09,NakTomTr19,VES2011}) given by
\begin{equation}\label{C_u}
  C_{\mathrm{crit}} =2 M^{1/2}\sqrt{M^{-1/2}KM^{-1/2}}M^{1/2}.
\end{equation}
The external damping matrix describes the (passive) dampers of the system. In our case it will depend on positive real parameters $v_i$, $i=1,\ldots,q$, $q\ll n$, (so called viscosities) and matrices corresponding to the positions of the dampers.

The vector function $x\colon [0,\infty)\to \mathbb{R}^n$ contains the state variables and $x_0$, $\dot{x}_0$ are the initial data. System disturbances are denoted by the vector function $u\colon [0,\infty)\to  \mathbb{R}^{m}$ and the matrix  $B_2\in \mathbb{R}^{n\times m }$. The output or the measurement vector function $w$ is determined by the output matrices $E_1, E_2 \in \mathbb{R}^{k\times n}$.

The problem of determining the optimal
damping matrix $C_{\mathrm{ext}}$ which will  ensure  optimal evanescence of
 the state $x$ from \eqref{MDK}
 is well studied. There is a vast literature in this field of research and this optimization problem has been intensively considered in the last two decades, see, e.g., \cite{BennerTomljTruh10, BRAB98, Nakic13, TRUH04, VES2011,
NakTomTr19, TRUHVES09,   Cox:04}. The minimization of vibrations was also intensively studied  in engineering and applied mathematics. Here, we list only a couple of references: \cite{beards1996structural, EIRivin2004, ITakewaki2009, du2016modeling, inman2017vibration,Gaw}.

For damping optimization there exist several  optimization criteria depending on different application areas.
An overview of such criteria can be found, e.g., in  \cite{VES2011}  or \cite{NAKIC02}. From the control theory perspective for the optimization criteria the $H_2$ or $H_\infty$ norms can be used. Within this setting, several authors considered model order reduction approaches in order to determine the optimal damping parameters  efficiently; for more details see  \cite{TomBeatGug18, BennerKuerTomljTruh15,AliBMSV17,BenV13e}.  Some criteria are based on eigenvalues, such as spectral abscissa criterion (for more details, see, e.g., \cite{FreitasLancaster99,NAKTOMTRUH13,WJSH2018}), while other criteria are based on the total energy of the system, such as the total average energy. Total average energy  was considered widely in the last two decades, see,  e.g.,  \cite{VES89,VES90,TRUHVES05,TrTomPuv17,Cox:04,Nakic13}. In \cite{BennerTomljTruh10,BennerTomljTruh11} the  authors   considered dimension reduction techniques that allowed efficient calculation of the total average energy.

In all the aforementioned papers, the time horizon of the system \eqref{MDK} was taken to be infinite. We are interested in the case of the finite time horizon, i.e.\ we study the system \eqref{MDK} in the time interval $[0,T]$, $T<\infty$. The infinite time horizon case is a natural choice in the cases when the vibration occurs over a longer period of time or the system is perpetually disturbed. But in the case of short duration phenomena such as earthquakes, finite time horizon is a much more suitable choice. From the mathematical point of view, the infinite time horizon leads to a computationally simpler optimization criterion, which is a very important aspect when designing the optimal damping structure. Hence if the finite time $T$ is large enough, it makes sense to choose $T=\infty$ even though this might lead to a (usually slightly) suboptimal design. But if the finite time is sufficiently small (but not too small as to render the design problem infeasible) then the choice $T=\infty$ is not viable and so the finite time horizon problem is the one
that merits a closer investigation.

Our choice of the optimization criterion will be based on the $p$--mixed $H_2$ norm, which was investigated in \cite{NakTomTr19} where the authors considered performance measure that takes into account the total average energy, and also the $H_2$ norm of the system. This criterion contains both the $H_2$ norm and the total average energy criteria as the special cases and takes into account both the initial data as well as external disturbances. This criterion will be also used in this paper as well, but instead of the infinite time horizon, we will consider the finite one.  In particular, the p-mixed $H_2$ norm of a system for the finite time horizon can be calculated by
\begin{equation*}
\trace E^\top E \left(\int_0^T \mathrm{e}^{A t} Z \mathrm{e}^{A^\top t} \,\mathrm{d}t \right),
\end{equation*}
where the matrix $A$ comes from the linearization of the system \eqref{MDK}, the matrix $Z$ encodes the information about dangerous external forces and initial conditions, and the matrix $E$ is given in terms of output matrices $E_1$, $E_2$. We will introduce this criterion in more details in the next section.

This objective function can be calculated directly using the following formula:
\begin{equation}\label{directFormula}
\trace{(X- e^{A T}Xe^{A^\top T})}, \text{ where $X$ is such that } A X+X A^\top=-Z.
\end{equation}
We would like to emphasize that the use of formula \eqref{directFormula} requires, besides solving Lyapunov equation, also calculating the matrix exponential, for more details see, e.g. \cite{DATTA04,ANT05,moler2003ndw,Golub2013,Higham05}.

The technique we develop can also be used in a more general setting of the finite horizon $H_2$ control in the case when the number of sensors and actuators is small, the system is parameter-dependent and one needs to calculate the appropriate $H_2$ norm for a large number of different parameters.

We do not pursue this line of research here to make the exposition less technical. Moreover our approach consists of two main parts; the offline part where we calculate matrices that do not depend on viscosities and the online part in which we organize calculations in such a way so that recycling of computationally demanding parts can be achieved efficiently.

One can also consider a model order reduction oriented approach for the finite time horizon problem which will result in a model of reduced order. This can be applied even for a large-scale systems, see, e.g. \cite{GoyalRed19, SinaniGug19, PK18} where authors considered model order reduction for a finite time horizon.  However, when the order of a system is reduced, then our new approach can again be employed for an efficient calculation of the $p$--mixed $H_2$ system norm.

Throughout the paper we will use the following notation. The symbol $\|\cdot \|$ denotes the standard vector norm or matrix 2-norm, depending on the context. If $p$ and $q$ are vectors, like in many programming languages including Julia or Matlab, the notation $A(p,q)$ will denote the submatrix of $A$ obtained by intersection of rows determined with elements of vector $p$ and columns determined with elements of vector $q$. Similarly, $i:k:j$ denotes the vector of integers from $i$ to $j$ with increments of $k$. For integers $j$ and $k$, we denote by $\delta_{j,k}$ the Kronecker delta symbol, i.e.\ $\delta_{j,k}=1$ if $j=k$, otherwise $\delta_{j,k}=0$. Also, for a matrix $A$ and a scalar $s$, $s-A$ denotes the matrix $sI-A$.

The paper is structured as follows. The finite time horizon $p$–mixed $H_2$ norm is introduced in section \ref{sec:criterion}. In section \ref{sec:derivation} we first derive formulae for auxiliary vectors $x_j$. Then,  in subsection \ref{subsec:The case of one-dimensional damping}, we investigate the case of one-dimensional damping and derive final formula for the finite time horizon $p$–mixed $H_2$ norm. Using the derived formulae we present an approach for an estimation of the finite time horizon $p$–mixed $H_2$ norm  in section \ref{sec: Estimation of the integral}. In particular, in subsection \ref{subsec:Estimation of the integration interval} we analyse estimation of the integration interval which is used in our algorithm for the calculation of the finite time horizon $p$–mixed $H_2$
norm presented in subsection \ref{subsec: Algorithm for norm}. In section \ref{sec: Numerical experiments} we illustrate the efficiency of the new approach through numerical experiments.


\end{section}

\begin{section}{Finite time horizon $p$--mixed $H_2$ norm}
\label{sec:criterion}
\begin{subsection}{Preliminaries}
\label{sub:preliminaries}

Differential equation in \eqref{MDK} can be transformed to the first order system in the phase space. For that purpose let $\Phi$ be a matrix which simultaneously diagonalizes $M$ and $K$, that is
\begin{equation}\label{MKdiag}
 \Phi^\top M \Phi = I ,
 \Phi^\top K \Phi = \Omega^2= \diag (\omega_1^2, \ldots, \omega_n^2 ) ,
 \end{equation}
where positive numbers $\omega_1, \omega_2,\ldots , \omega_n$ are undamped eigenfrequencies of the system, i.e.\ square roots of the eigenvalues corresponding to the system with $C=0$ (that is, the eigenvalues of $Q(\lambda) = \lambda^2 M - K$). In this case the matrix $\Phi$ diagonalizes the internal damping matrix defined by \eqref{C_u}, that is $\Phi^\top C_{\mathrm{int}} \Phi =\nu \Omega,$ for $\nu>0$, with $\nu=2\alpha$.
 Now, using the substitutions   $y_1=\Omega \Phi^{-1}x$ and
$y_2=\Phi^{-1}\dot{x} $, the differential equation in \eqref{MDK} can be written as

\begin{equation}\label{eq:first_order_system}
	\dot{y}=A y +  B u,\; y(0)  = y_0,
\end{equation}
where
\begin{equation}\label{A form}
  A=\begin{bmatrix}
  0 & \Omega \\
  -\Omega & -\nu \Omega -D \\
\end{bmatrix}%
,\quad
  B=\begin{bmatrix}
  0  \\
 \Phi^\top B_2  \\
\end{bmatrix},
\quad
 y=
\begin{bmatrix}
  y_1 \\
  y_2 \\
\end{bmatrix},
\end{equation}
 with $D=\Phi^\top C_{\mathrm{ext}} \Phi $.    The output is determined by
\[
w =  E y\quad \text{ with } \quad  E =\begin{bmatrix}
E_1\Phi\Omega^{-1} & 0 \\ 0 & E_2\Phi
\end{bmatrix}.
\]
 Here $y_0$ contains the corresponding transformation of the initial data,
 see, e.g. \cite{BRAB98,  VES2011, TRUHVES09}.  This is the so-called \emph{modal representation} of the system \eqref{MDK}.

The solution of \eqref{MDK} hence can be written as
\begin{equation}
\label{eq:formula_za_w}
w(t) = E \mathrm{e}^{At}y_0 +  E \int_0^t \mathrm{e}^{A(t - \tau)}
B u(\tau) \,\mathrm{d} \tau.
\end{equation}

Let the parameter $p$ satisfy $0 \le p \le 1$.
The $p$--mixed $H_2$ norm of a system $G$, denoted by $\lVert G \rVert_{2,p} $, is defined as
\begin{equation}\label{eq:pmixh2norm}
\lVert G \rVert_{2,p}^2 = (1-p) \lVert G \rVert_2^2 + p \lVert G \rVert_{2,\mathrm{hom}}^2 .
\end{equation}
In \eqref{eq:pmixh2norm} $\lVert \cdot \rVert_2 $ denotes the standard $H_2$ norm given by
\begin{equation*}
\lVert G \rVert_2^2 = \frac{1}{2\pi} \int_{- \infty}^{\infty} \trace (\hat G(i \omega))^* \hat G(i \omega)) \,\mathrm{d} \omega ,
\end{equation*}
where $\hat G$ denotes the transfer function of the system $G$, i.e.\ the Laplace transform of the mapping $u\mapsto w$.
It can be shown, see e.g.\ \cite{dullerud2013course}, that the formula for $\lVert G \rVert_2$ in the time domain is given by
\begin{equation*}
\lVert G \rVert_2^2 = \trace \left(\int_0^\infty  E^\top\mathrm{e}^{At}BB^\top \mathrm{e}^{A^\top t}E \,\mathrm{d} t \right).
\end{equation*}
With $\lVert \cdot \rVert_{2,\mathrm{hom}}$, in \eqref{eq:pmixh2norm}, we denote the $H_2$ norm of the corresponding homogeneous ($u=0$) system given by
\begin{subequations}
\begin{align*}
\lVert G \rVert_{2, \mathrm{hom}}^2 &=  \int_{\lVert y_0 \rVert^2  = 1 } \int_0^{\infty} \lVert w(t;y_0) \rVert^2 \, \mathrm{d}
t \, \mathrm{d} \sigma   \overset{\eqref{eq:formula_za_w}}{=}  \int_{\lVert y_0 \rVert^2  = 1 } y_0^\top \left( \int_0^{\infty}  \mathrm{e}^{A^\top t} E^\top E \mathrm{e}^{At} \, \mathrm{d}
t \right) \, y_0  \, \mathrm{d} \sigma,
\end{align*}
\end{subequations}
where $\sigma$ is an averaging (surface) measure on the unit sphere $\mathbb{R}^{2n}$. More precisely, for a given measure in $\mathbb{R}^{2n}$, the corresponding surface measure $\sigma$ is given by the Minkowski–Steiner formula \cite{federer1969}. Both $\lVert \cdot \rVert_2 $ and $\lVert \cdot \rVert_{2, \mathrm{hom}}$ can be expressed in terms of solutions of Lyapunov equations; for details see \cite{NakTomTr19}.
The $p$--mixed $H_2$ norm can be calculated as (see \cite[eq.\ (14)]{NakTomTr19}):
\begin{equation}
	\label{eq:trace_abs}
\trace (E^\top EX), \text{ where $X$ is such that }  A X + X  A^\top = - p Z_\sigma - (1-p)B B^\top,
\end{equation}
where the matrix $Z_\sigma$ depends on the choice of averaging measure $\sigma$ on the set of unit initial data.

\end{subsection}
\begin{subsection}{Definition of the finite time horizon $p$--mixed $H_2$ norm}
\label{sub:definition_of_}

The Lyapunov equation in \eqref{eq:trace_abs} occurs, in both the $H_2$ term and the homogeneous $H_2$ term, from the integrals in the time domain of the form $\int_0^{\infty} \mathrm{e}^{A t} \boldsymbol{\cdot} \mathrm{e}^{A^\top t} \,\mathrm{d}t$. It is easy to see that in the case of the finite time horizon the formula corresponding to \eqref{eq:trace_abs} is
\begin{equation}
\label{eq:objectiveFun}
\trace \left( E^\top E\int_0^T \mathrm{e}^{A t} (p Z_\sigma + (1-p)B B^\top ) \mathrm{e}^{A^\top t} \,\mathrm{d}t \right) .
\end{equation}
It is well known and used frequently in computations that such an expression again can be written in terms of a Lyapunov equation, and so the last expression can be written as (see, e.g., \cite{ANT05,DATTA04,NakTomTr19})
\begin{equation}\label{eq:lyapeq1}
    \trace \left( E^\top E\tilde X \right),
\end{equation}
where $X$ is such that $A \tilde X + \tilde X A^\top =  \mathrm{e}^{A T} (pZ_\sigma + (1-p)B B^\top ) \mathrm{e}^{A^\top T} - p Z_\sigma - (1-p)B B^\top $.

Indeed, the function $X(t) = \mathrm{e}^{At}(p Z_\sigma + (1-p)B B^\top )\mathrm{e}^{A^\top t}$ is the solution of the Cauchy problem
\begin{equation*}
\dot{\tilde X}(t) = A \tilde X(t) + \tilde X(t)A^\top , \; \tilde X(0) = p Z_\sigma + (1-p)B B^\top ,
\end{equation*}
and by integrating the differential equation from $0$ to $T$ we obtain
\begin{equation*}
\tilde X(T) - \tilde X(0) = A \int_0^T \tilde X(t) \,\mathrm{d}t + \int_0^T \tilde X(t) \,\mathrm{d}t A^\top .
\end{equation*}
If we denote by $X$ the solution of the Lyapunov equation
$$AX + XA^\top = - p Z_\sigma - (1-p)B B^\top, $$ it is easy to see that
$\tilde X = X - \mathrm{e}^{A T} X \mathrm{e}^{A^\top T}$, hence \eqref{eq:lyapeq1} can be written as
\begin{equation}
	\label{eq:lyapeq2}
	\trace \left( E^\top E \left( X - \mathrm{e}^{A T} X \mathrm{e}^{A^\top T} \right)  \right) ,
\end{equation}
where $X$ is such that $AX + XA^\top = - p Z_\sigma - (1-p)B B^\top $.
By duality, \eqref{eq:lyapeq2} is equivalent to
\begin{equation}
	\label{eq:lyapeq1_dual}
\trace \left( (pZ_\sigma + (1-p)B B^\top ) \left( X - \mathrm{e}^{A T} X \mathrm{e}^{A^\top T} \right)  \right) ,
\end{equation}
where $X$ is such that $A^\top X + XA = - E^\top E$.

Although \eqref{eq:lyapeq1_dual} is much easier to compute than directly \eqref{eq:objectiveFun} in the case of one system, we will show that in some instances \eqref{eq:objectiveFun} has an advantage. More precisely, if $E$ has a simple structure, $Z_\sigma$ and $B$ are low rank matrices and the goal is to compute \eqref{eq:objectiveFun} (or \eqref{eq:lyapeq1_dual}) for a large number of external low rank damping matrices $C_{\mathrm{ext}}$, which is a frequent situation when optimizing $C_{\mathrm{ext}}$, the expression \ref{eq:objectiveFun} can be used to construct an efficient algorithm for such a task.

In the sequel we will consider the typical case when
\begin{align}\label{matrices EZ}
E^\top E = \frac{1}{2} I  \quad \mbox{and} \quad    pZ_\sigma + (1-p)B B^\top  = Z,\\
\mbox{ with}\quad
\label{eq:matrixZ}
Z  = \begin{bmatrix}
p Z_1 & 0 \\ 0 & Z_1
\end{bmatrix} \quad \text{and} \quad
Z_1 = \begin{bmatrix}
I_r & 0 \\ 0 & 0
\end{bmatrix}.
\end{align}
Indeed, when modeling vibrational systems, the matrix $B_2$ is usually designed as a band--pass filter where only the dangerous frequencies are passed through. In terms of the modal representation \eqref{eq:first_order_system}-\eqref{A form}, this would mean that $ B_2$ has the form $B_2= Z_1$, where the dangerous frequencies are $\omega_1,\ldots, \omega_r$. Typically $r$ is a much smaller number than $n$ which will also be beneficial for our approach. For the dangerous frequencies we choose those which have a significant influence on the behavior of the system, e.g., those that may lead to system resonances. In damping applications it is typical to only damp dangerous frequencies, hence our choice of the matrix $B_2$.
The measure $\sigma$ typically is chosen in such a way that it attenuates initial data which are not dangerous. In particular, the surface measure can be chosen in such a way that it is generated by the Lebesgue measure on the subspace spanned by the vectors $[x_i, 0]^{\top}$ and $[0,x_i]^{\top}$, $i=1,\ldots,r$, where $x_i$ are the eigenvectors of $\omega_i$, and on the rest of $\mathbb{R}^{2n}$ is generated by the Dirac measure concentrated at zero. Then, in the modal representation we obtain $ Z_\sigma = c\mathop{\mathrm{diag}}(Z_1,Z_1)$, where the constant $c>0$ only depends on the dimension. This corresponds to ignoring the initial data spanned by non-dangerous frequencies. Since the squared $L_2$ norm of the output, $\lVert w \rVert_2^2 $, typically corresponds to the energy of the whole system, in the modal representation we obtain $E^\top E = \frac{1}{2} I$. See \cite[Section 4]{NakTomTr19} for the detailed explanation. Hence, by appropriate scaling, we obtain \eqref{matrices EZ}-\eqref{eq:matrixZ}. Moreover, in the case  when $\lVert w \rVert_2^2 $ equals the energy corresponding to dangerous frequencies, we obtain $E^\top E = \frac{1}{2} Z$,  which can be also covered by our procedure; see \eqref{rem:do_r} below.

Therefore, our goal is to efficiently calculate
\begin{equation}
\label{objectiveFun}
\trace \left(\int_0^T \mathrm{e}^{A t} Z \mathrm{e}^{A^\top t} \,\mathrm{d}t \right),
\end{equation}
with $A$ given in \eqref{A form} and $Z$ given in \eqref{eq:matrixZ}.

Our technique is not limited to this particular choice of matrices $B_2$, $E_1$, $E_2$ and the measure $\sigma$. As long as the corresponding matrix $pZ_\sigma + (1-p)B B^\top$ is of low rank and the matrix $E^\top E$ has a simple structure, one can construct a similar procedure. We limit our attention to this particular case to not overburden the paper with technicalities.

\end{subsection}
\end{section}

\begin{section}{Derivation of the formula for the finite time horizon $p$--mixed $H_2$ norm} 
\label{sec:derivation}
In this section we use \eqref{objectiveFun} directly to obtain a formula for the finite time horizon $p$--mixed $H_2$ norm.
For the purpose of the further structured calculations, we will use the fact that $A$ can be written as
\begin{align}
A &= \begin{bmatrix}
0 & \Omega \\ - \Omega & - \nu \Omega - D
\end{bmatrix} = \begin{bmatrix}
0 & \Omega \\ - \Omega & - \nu \Omega
\end{bmatrix} - \begin{bmatrix}
0 & 0 \\ 0 & D
\end{bmatrix}
=
 A_0 - A_1. \label{DefMatrixA}
\end{align}

Using the fact that the matrix exponential function is the inverse Laplace transform of the corresponding resolvent, we have
\begin{subequations}
\begin{align*}
&\trace \left( \int_0^T \mathrm{e}^{ A t }Z \mathrm{e}^{ A^{\top} t} \,\mathrm{d}t\right)  = \int_0^T \trace ((\sqrt{Z} \mathrm{e}^{ A^{\top} t })^\top (\sqrt{Z} \mathrm{e}^{ A^{\top} t})) \,\mathrm{d}t \\
&= \sum_{j,k=1}^{2n} \int_0^T  \left((\sqrt{Z} \mathrm{e}^{ A^{\top} t})_{j,k}\right)^2  \,\mathrm{d}t  \\
&= \sum_{j,k=1}^{2n} \int_0^T  \left(\frac{1}{2 \pi \mathrm{i}}\left(\sqrt{Z} \int_{-\mathrm{i} \infty}^{+\mathrm{i} \infty} \mathrm{e}^{\lambda t} (\lambda - A^{\top} )^{-1} \,\mathrm{d} \lambda\right)_{j,k}\right)^2  \,\mathrm{d}t \\
&= \sum_{j,k=1}^{2n} \int_0^T  \left(\frac{1}{2 \pi}\left( \int_{- \infty}^{+ \infty} \mathrm{e}^{\mathrm{i} s t} \sqrt{Z}(\mathrm{i}s - A^{\top} )^{-1} \,\mathrm{d} s\right)_{j,k}\right)^2  \,\mathrm{d}t.
\end{align*}
\end{subequations}
Due to the structure of the matrix $Z$, the summation index of $j$ goes from $1$ to $r$ and from $n+1$ to $n+r$, so we obtain
\begin{multline}\label{eq:sums}
\trace \left( \int_0^T \mathrm{e}^{ A t }Z \mathrm{e}^{ A^{\top} t} \,\mathrm{d}t\right)  = \frac{p}{4 \pi^2}  \sum_{j=1}^r \sum_{k=1}^{2n} \int_0^T  \left( \int_{- \infty}^{+ \infty}  h_{jk}(t,s) \,\mathrm{d} s \right)^2  \,\mathrm{d}t \\
+\frac{1}{4 \pi^2} \sum_{j=n+1}^{n+r} \sum_{k=1}^{2n}  \int_0^T  \left( \int_{- \infty}^{+ \infty} h_{jk}(t,s)\,\mathrm{d} s \right)^2  \,\mathrm{d}t ,
\end{multline}
where
\begin{equation}\label{h_{jk}(t,s)}
h_{jk}(t,s)=(\cos st + \mathrm{i} \sin st) e_j^\top(\mathrm{i}s - A^{\top} )^{-1}e_k,
\end{equation}
and $e_j$ denotes the $j-$th canonical vector in $\mathbb{C}^n$, $\mathbb{C}^{2n}$, $\mathbb{R}^n$ or $\mathbb{R}^{2n}$, depending on the context.

\begin{remark}
\label{rem:do_r}
	Note that if we would have chosen $E^\top E=\frac{1}{2} Z$, then the sums in \eqref{eq:sums} would be given by
	$p  \sum_{j,k=1}^r + \sum_{j,k=n+1}^{n+r} $ and hence the Algorithm \ref{algorithm1} from subsection \eqref{subsec: Algorithm for norm} can be easily modified to cover this case as well.
\end{remark}

To construct an efficient algorithm for the calculation of the last sum, we will carefully study the terms $e_j^\top(\mathrm{i}s - A^{\top} )^{-1}e_k$ using the structure of the matrix $A$, distinguishing those terms which do not depend on the $C_{\mathrm{ext}}$.

We calculate
\begin{subequations}
\begin{align*}
(\mathrm{i}s - A^{\top} )^{-1} &= (\mathrm{i}s - A_0^{\top} +A_1)^{-1} = \left( (\mathrm{i}s -A_0^{\top}) (I + (\mathrm{i}s -A_0^{\top})^{-1} A_1) \right)^{-1} \\
&= \left( I + (\mathrm{i}s -A_0^{\top})^{-1} A_1 \right)^{-1}   (\mathrm{i}s -A_0^{\top})^{-1},
\end{align*}
\end{subequations}
hence
\begin{subequations}
\begin{align*}
e_j^\top(\mathrm{i}s - A^{\top})^{-1}e_k &=  e_j^\top\left( I + (\mathrm{i}s -A_0^{\top})^{-1} A_1  \right)^{-1}   (\mathrm{i}s -A_0^{\top})^{-1}e_k \\
&= \left( ( I + (\mathrm{i}s -A_0^{\top})^{-1} A_1)^{-\top}   e_j\right)^\top   (\mathrm{i}s -A_0^{\top})^{-1}e_k \\
&= \left( ( I + A_1(\mathrm{i}s -A_0)^{-1} )^{-1}  e_j\right)^\top   (\mathrm{i}s -A_0^{\top})^{-1}e_k .
\end{align*}
\end{subequations}
Let  $x_j = ( I + A_1(\mathrm{i}s -A_0)^{-1} )^{-1}  e_j$ i.e.\ $(I + A_1(\mathrm{i}s -A_0)^{-1} ) x_j = e_j$ . Let  $x_j = [ x_j^1 \; x_j^2 ]^\top$. Note that $e_j^\top(\mathrm{i}s -A^{\top})^{-1}e_k$ is a scalar product of two vectors, vector $x_j$ depends on $s$, viscosities and damping positions and vector $(\mathrm{i}s -A_0^{\top})^{-1}e_k$ depends only on $s$. Moreover, we would like to derive a formulae for $x_j$ which will be considered in the next subsection.


\begin{subsection}{Calculation of the vectors $x_j$}
\label{subsec:Calculation of the vectors x_j}

From
\begin{equation*}
\mathrm{i}s -A_0 = \begin{bmatrix}
\mathrm{i}s & - \Omega \\ \Omega & \mathrm{i}s + \nu \Omega
\end{bmatrix}
\end{equation*}
we obtain
\begin{equation*}
(\mathrm{i}s -A_0 )^{-1} = \begin{bmatrix}
(\nu \Omega + \mathrm{i}s ) L(s) & \Omega L(s) \\ - \Omega L(s) & \mathrm{i}s L(s)
\end{bmatrix},	
\end{equation*}	
where
\begin{equation}\label{matrixLs}
L(s) = \left( \Omega^2 + \mathrm{i} s \nu \Omega - s^2 \right)^{-1}=F(s)-iG(s),	
\end{equation}
where matrices $F(s)$ and $-G(s)$ are real and imaginary part of matrix $L(s)$. Since the matrix $ \Omega^2 + \mathrm{i} s \nu \Omega - s^2$ is a diagonal complex matrix, its inverse can be calculated directly. Now we have
\begin{subequations}
\begin{align*}
I + A_1(\mathrm{i}s -A_0)^{-1}  &= \begin{bmatrix}
I & 0 \\ 0 & I
\end{bmatrix} + \begin{bmatrix}
0 & 0 \\ 0 & D
\end{bmatrix}
\begin{bmatrix}
(\nu \Omega + \mathrm{i}s ) L(s) &  \Omega L(s) \\  - \Omega L(s) & \mathrm{i}s L(s)
\end{bmatrix}\\
&= \begin{bmatrix}
I & 0 \\ - D \Omega L(s) & I + \mathrm{i}s D L(s)
\end{bmatrix}.
\end{align*}
\end{subequations}
For $ 1\le j \le r$  we obtain
\begin{equation*}
x_j^1 = e_j, \; \left( I + \mathrm{i}s D L(s) \right)x_j^2 =  D \Omega L(s) e_j
\end{equation*}
and for $n+1\le j\le n+r$
\begin{equation*}
x_j^1 = 0, \; \left( I + \mathrm{i}s D L(s) \right)x_j^2 =  e_{j-n} .
\end{equation*}
In the following proposition we obtain the formulae for the real and imaginary parts of complex vectors $x_j^2$, for $j=1\ldots,r,n+1,\ldots,n+r$, in terms of the solutions of \emph{real} linear systems \eqref{eq:sustav1} and \eqref{eq:sustav2}.
\begin{proposition}\label{systems}
Let matrix $L(s)$ be given as in \eqref{matrixLs} and let $x_j^2 = x_j^{\Re} + \mathrm{i} x_j^{\Im}$, with $x_j^{\Re}, x_j^{\Im} \in \mathbb{R}^n$.
Then systems $\left( I + \mathrm{i}s D L(s) \right)x_j^2 =  D \Omega L(s) e_j$, for $ 1\le j \le r$, and $\left( I + \mathrm{i}s D L(s) \right)x_j^2 =  e_{j-n}$, for $n+1\le j\le n+r$, are equivalent to the following systems
\begin{footnotesize}
\begin{multline}\label{eq:sustav1}
    \begin{bmatrix}
	I & -s \left( I + s D G(s) \right)^{-1} D F(s) \\ 0 & s^2 D F(s) \left( I + s D G(s) \right)^{-1} D F(s) + I + s D G(s)
	\end{bmatrix}
	\begin{bmatrix}
	x_j^{\Re} \\ x_j^{\Im}
	\end{bmatrix} \\
	= \begin{bmatrix}
	 \left( I + s D G(s) \right)^{-1} D \Omega F(s) e_j \\- s D F(s) \left( I + s D G(s) \right)^{-1} D \Omega F(s) e_j - D \Omega G(s) e_j
	\end{bmatrix},
\end{multline}
\begin{multline}\label{eq:sustav2}
    \begin{bmatrix}
	I & -s \left( I + s D G(s) \right)^{-1} D F(s) \\ 0 & s^2 D F(s) \left( I + s D G(s) \right)^{-1} D F(s) + I + s D G(s)
	\end{bmatrix}
	\begin{bmatrix}
	x_j^{\Re} \\ x_j^{\Im}
	\end{bmatrix} \\
	= \begin{bmatrix}
	\left( I + s D G(s) \right)^{-1} e_{j-n} \\ -s D  F(s) \left( I + s D G(s) \right)^{-1} e_{j-n}
	\end{bmatrix},
\end{multline}
\end{footnotesize}
respectively.
\end{proposition}
Proof of this proposition is given in Appendix \ref{Appendix}.

In the next subsection we analyse formulae for the case of one-dimensional damping which will allow us to construct an efficient procedure for the calculation of the finite time horizon $p$--mixed $H_2$ norm.
\end{subsection}

\begin{subsection}{The case of one-dimensional damping}
\label{subsec:The case of one-dimensional damping}

The systems of linear equations given above can be solved in a more general setting. But in this paper, to reduce the technicalities, we only treat the particular case when there is just one damper of dimension one, and the only parameter is its viscosity. The procedure can be straightforwardly extended to the multi-parameter case in the case of a small number of parameters. Then the damping matrix $D$ has the following form
\begin{equation}\label{Cext}
    D=\Phi^\top C_{\mathrm{ext}} \Phi = v \Phi^\top e e^\top \Phi = \gamma UU^\top,
\end{equation}
where the vector $e$ encodes the position of the damper.
The parameter $\gamma >0$ is a product of viscosity parameter $v$ and the 2-norm of the matrix $\Phi^\top e e^\top \Phi$ and vector $U\in \mathbb{R}^{n \times 1}$ determines the geometry of the damping position.


In the following proposition we give explicit solution of the equations \eqref{eq:sustav1}-\eqref{eq:sustav2} in terms of the parameter $\gamma$ that determines the viscosity parameter.
\begin{proposition}\label{prop:solutionxj}
  Let  $q=1$ and assume that $U$ and $\gamma$ define  damping matrix $D$ as in \eqref{Cext}. Furthermore, let $g(s)=U^\top G(s)U=\sum_{j=1}^n u_j^2g_j(s)$ and $f(s)=U^\top F(s)U=\sum_{j=1}^n u_j^2f_j(s)$, where $f_j(s)$ and $g_j(s)$, for $j=1,\ldots, n$ are diagonal elements of $F(s)$ and $G(s)$ given by \eqref{matrixLs}, respectively.    Then, the solution of \eqref{eq:sustav1} is given by ($1\le j \le r$)
\begin{align}
	x_j^\Im (s) &= a(s,\gamma)  U, \label{xj1}\\
	x_j^\Re (s) &= \left( \frac{\gamma f_j(s) u_j \omega_j}{1+s \gamma g(s)} + \frac{s\gamma f(s)}{1+s\gamma g(s)}a(s,\gamma) \right)U,\label{xr1}
\end{align}
where
\begin{multline}
    a(s,\gamma)=- \gamma u_j \omega_j \times\\ \times \frac{g_j(s)+ s \gamma \left(f(s)f_j(s)+2g(s)g_j(s)\right)+ s^2\gamma^2g(s)\left(f(s)f_j(s)+g(s) g_j(s)\right)}{1+ 3 s \gamma g(s) +s^2\gamma^2 \left(3g(s)^2+f(s)^2\right)+s^3\gamma^3 g(s)\left(g(s)^2+f(s)^2)\right)},
\end{multline}
and the solution of \eqref{eq:sustav2} is given by ($n+1\le j \le n+r$)
\begin{align}
	x_j^\Im (s) &= b(s,\gamma)  U, \label{xj2}\\
	x_j^\Re (s) &= e_{j-n} + \left(-\frac{s\gamma u_{j-n}g_{j-n}(s)}{1+s\gamma g(s)} + \frac{s\gamma f(s)}{1+s\gamma g(s)}b(s,\gamma) \right)U,\label{xr2}
\end{align}
where
\begin{equation}
    b(s,\gamma)=-s\gamma u_{j-n} \frac{f_{j-n}(s)+s\gamma(g(s)f_{j-n}(s)-f(s)g_{j-n}(s))}{(1+s\gamma  g(s))^2+(s \gamma f(s))^2}.
\end{equation}
\end{proposition}
Proof of this proposition is given in Appendix \ref{Appendix}.

Obviously, $ x_j^{\Re}$ and $x_j^{\Im}$ depend on $s$ but not on $t$ and sometimes, to emphasise this, we write $ x_j^{\Re} (s)$ and $x_j^{\Im}(s)$.

The following proposition gives formulae for  $h_{jk}(t,s)$ defined by  \eqref{h_{jk}(t,s)}, which are our main target in light of \eqref{eq:sums}.
\begin{proposition}\label{prop:solutionHjk}
Let all assumptions from \eqref{prop:solutionxj} hold. Then, for $h_{jk}(t,s)$, defined by  \eqref{h_{jk}(t,s)}, we have the following formulae.\\
For $1\le j \le r$ and $1\le k\le n$ we have
\begin{equation*}
h_{jk}(t,s) =
2  \cos st \cdot  \left( \left(\nu \omega_k f_k(s) + s g_k(s)\right)\delta_{j,k} + \omega_k \left( f_k(s) (x_j^{\Re})_k +  g_k(s) (x_j^{\Im})_k\right) \right) .
\end{equation*}
For $1\le j \le r$ and $n+1\le k\le 2n$ we have
\begin{equation*}
	h_{jk}(t,s) =2\cos  st \cdot \left(- \omega_{k-n} f_{k-n}(s) \delta_{j,k-n} + s g_{k-n} (s) (x_j^{\Re})_{k-n} - s f_{k-n} (s) (x_j^{\Im})_{k-n} \right) .
\end{equation*}
For $n+1\le j \le n+r$ and $1\le k\le n$ we have
\begin{align*}
h_{jk}(t,s) &=  2  \cos st \cdot \left( \omega_k \left( f_k(s) (x_j^{\Re})_k + g_k(s)(x_j^{\Im})_k \right)  \right).
\end{align*}
For $n+1\le j \le n+r$ and $n+1\le k\le 2n$ we have
\begin{align*}
h_{jk}(t,s) &= 2s  \cos st \cdot \left( g_{k-n}(s) (x_j^{\Re})_{k-n}  - f_{k-n} (s) (x_j^{\Im})_{k-n} \right) .
\end{align*}
\end{proposition}
Proof of this proposition is given in Appendix \ref{Appendix}.

Finally, using derived explicit  formulae for $h_{jk}$ we are able to write a formula for the finite time horizon $p$--mixed $H_2$ norm, that is, we calculated all the ingredients of the formula \eqref{eq:sums}.

In this section we only considered formulae for the one-parameter case, meaning that $U\in \mathbb{R}^{n \times q}$ is a vector. These formulae can be extended straightforwardly to the multi-parameter case, which can still be used for an efficient p-mixed $H_2$ norm calculation as long as we have a small number of parameters, meaning $q\ll n$. In general, the matrix $U$ would contain $q>1$ columns, and the Sherman–Morrison–Woodbury formula would include the inverses of $q \times q$ matrices. The obtained formulae would have a similar structure, which would lead to an algorithm of the same structure as the one given in \ref{algorithm1}. We have considered only the one-parameter case to simplify the exposition, as in the general case the formulae are more complicated.

Now, in the next section we explain how to use the quadrature rule in order to calculate an approximation of the expression above.

\end{subsection}

\end{section}


\begin{section}{Estimation of the integrals}
\label{sec: Estimation of the integral}

Our approach is based on a numerical integration of a very structured and oscillatory  function.  This  is a widely investigated field and overview of some methods can be found in  \cite{DHI18,2004orthogonal,Cohen07}. In order to be able to recycle data and use our formulae efficiently we will use Gauss quadrature rule for the integration with respect to time and adaptive Simpson rule (see, e.g., \cite{2004orthogonal}) for the highly oscillatory part.
In particular, we will use the following estimate for \eqref{eq:sums}
\begin{multline}\label{estint}
\trace \left( \int_0^T \mathrm{e}^{ A t }Z \mathrm{e}^{ A^{\top} t} \,\mathrm{d}t\right) \approx  \frac{p}{4 \pi^2}  \sum_{j=1}^r \sum_{k=1}^{2n}  \sum_{\alpha =1}^{n_t} \eta_\alpha \left( \sum_{\beta =1}^{n_s} \zeta_\beta h_{jk}(t_\alpha,s_\beta)  \right)^2 \\+\frac{1}{4 \pi^2}  \sum_{j=n+1}^{n+r} \sum_{k=1}^{2n}   \sum_{\alpha =1}^{n_t} \eta_\alpha \left( \sum_{\beta =1}^{n_s} \zeta_\beta h_{jk}(t_\alpha,s_\beta)  \right)^2 ,
\end{multline}
where $\{(t_\alpha, \eta_\alpha)\}_{\alpha =1}^{n_t}$ are the nodes and weights for the integral over the time variable $t$ and $\{(s_\beta, \zeta_\beta)\}_{\beta =1}^{n_s}$ are the nodes and weights for the integral over the frequency variable $s$.

For the estimation of $\int_{- \infty}^{+ \infty} h_{jk}(t,s) \,\mathrm{d} s$ we will use combination of Simpson rule and adaptive Simpson rule for the  highly oscillatory part.
First we estimate the indefinite integral by its finite approximation
\begin{equation*}
\int_{- \infty}^{+ \infty} h_{jk}(t,s) \,\mathrm{d} s = 2\int_{- \infty}^{0} h_{jk}(t,s) \,\mathrm{d} s \approx  2\int_{- S_{\max}}^{0} h_{jk}(t,s) \,\mathrm{d} s.
\end{equation*}
In the next subsection   we will derive an upper bound which shows how large parameter $S_{\max}$ should be.


\begin{subsection}{Estimation of the integration interval}
\label{subsec:Estimation of the integration interval}

Since the leading term in our integral is $  e_j^\top (\mathrm{i}s-A^\top )^{-1} e_k ,$ for $j=1,\ldots,r,n+1,\ldots, n+r$, $k=1,\ldots, 2n$, we will try to determine efficiently how large (and small) values of the parameter $s$ should be considered. Assume that the parameter $\gamma$ is taken from the range $[0, \gamma_{\max}]$ and let $\omega_{\max} = \max \left\{ \omega_1,\ldots, \omega_n \right\}$ .

By using the structure of the matrix $A$ given by \eqref{DefMatrixA} we have
\begin{equation*}
    (\mathrm{i}s-A^\top)^{-1}   =
  \left( I + (\mathrm{i}s -A_0^\top)^{-1} A_1 \right)^{-1}
   (\mathrm{i}s -A_0^\top)^{-1}.
\end{equation*}
When  $\|(\mathrm{i}s -A_0^\top)^{-1} A_1\|<1$ we have that
$ I + (\mathrm{i}s -A_0^\top)^{-1} A_1$ is a non-singular matrix and we have $\| \left( I + (\mathrm{i}s -A_0^\top)^{-1} A_1\right)^{-1}  \|\leq  \left(1 -\| (\mathrm{i}s -A_0^\top)^{-1}A_1\|\right)^{-1} $. If we use that $\| A_1\| \leq \gamma$, we get
\begin{equation}\label{ocjena1}
  \|   (\mathrm{i}s-A^\top)^{-1}   \| \leq
 \frac{\|(\mathrm{i}s-A_0^\top)^{-1}\|}{1-\gamma \|(\mathrm{i}s-A_0^\top)^{-1}\|}.
\end{equation}
We want to evaluate the norm of the matrix $(\mathrm{i}s-A_0^\top)^{-1}$, for which it is sufficient to calculate the eigenvalues of the matrix $(\mathrm{i}s-A_0^\top)^{-1}(-\mathrm{i}s-A_0 )^{-1}$. Since
\begin{align*}
    (\mathrm{i}s-A_0^\top)^{-1}(-\mathrm{i}s-A_0 )^{-1} &= (s^2+\mathrm{i}s(A_0^\top-A_0 )+A_0^\top A_0)^{-1} \\
    &= \begin{bmatrix}
    s^2+\Omega^2  & - 2\mathrm{i}s\Omega -\nu \Omega^2\\
     2\mathrm{i}s\Omega - \nu \Omega^2  & s^2+(1+\nu^2)\Omega^2
    \end{bmatrix}^{-1},
\end{align*}
it is sufficient to investigate the eigenvalues of $B^{-1}$, where
\begin{equation*}
    B=  \begin{bmatrix}
    s^2+\omega^2  & - 2\mathrm{i}s\omega -\nu \omega^2\\
    2\mathrm{i}s\omega - \nu \omega^2  & s^2+(1+\nu^2)\omega^2
    \end{bmatrix} .
\end{equation*}
We have
\begin{equation*}
    \det(B-\lambda I)=(s^2+\omega^2-\lambda)(s^2+\omega^2+\nu^2\omega^2-\lambda)-(2\mathrm{i}s\omega-\nu\omega^2)(-2\mathrm{i}s\omega-\nu\omega^2)=0.
\end{equation*}
We can write this as
\begin{equation*}
    \mu(\mu+\nu^2\omega^2)-(\nu^2\omega^4+4s^2\omega^2)=0,
\end{equation*}
for $\mu=s^2+\omega^2-\lambda$. This is the quadratic equation in the variable $\mu$ and the solutions are given by
\begin{equation*}
    \mu_{1,2}=\frac{-\nu^2 \omega^2 \pm \sqrt{\nu^4\omega^4+4(\nu^2\omega^4+4s^2\omega^2)}}{2}.
\end{equation*}
Hence the eigenvalues of matrix $B$ are given by
\begin{equation*}
    \lambda_{1,2}=\frac{\nu^2 \omega^2 \mp \sqrt{\nu^4\omega^4+4(\nu^2\omega^4+4s^2\omega^2)}}{2}+s^2+\omega^2.
\end{equation*}
If we want $\|(\mathrm{i}s-A_0^\top)^{-1}\|<\delta$ for some tolerance $\delta>0$, we must have
\begin{equation*}
\frac{\displaystyle 1}{\displaystyle \sqrt{\lambda_{1,2}}}<\delta \Rightarrow \lambda_{1,2} \geq \frac{\displaystyle 1}{\displaystyle \delta^2},
\end{equation*}
for all $\omega = \omega_1,\ldots, \omega_n$.
Note that
\begin{equation*}
    \lambda_{1,2}\geq \frac{\nu^2 \omega^2 - \nu^2 \omega^2 -2\nu\omega^2 - 4s\omega}{2}+s^2+\omega^2=(s-\omega)^2-\nu\omega^2.
\end{equation*}
To obtain $\|(\mathrm{i}s-A^\top)^{-1}\|<\varepsilon$ for a tolerance $\varepsilon >0$, note that from \eqref{ocjena1} it follows that it is sufficient to have
\begin{equation*}
    \|(\mathrm{i}s-A_0^\top)^{-1}\|\leq \frac{\varepsilon}{1+\gamma \varepsilon}.
\end{equation*}
Now we can take $\delta=\frac{\displaystyle \varepsilon}{\displaystyle 1+\gamma \varepsilon}$, so it follows
\begin{equation*}
(s-\omega)^2-\nu\omega^2 \geq \frac{(1+\gamma \varepsilon)^2}{\varepsilon^2}
\end{equation*}
which implies
\begin{equation*}
    s^2-2\omega s + (1-\nu)\omega^2 - \frac{(1+\gamma \varepsilon)^2}{\varepsilon^2} \geq 0.
\end{equation*}
This is a quadratic inequality in the variable $s$ which is satisfied for all $s$ for which we have
\begin{equation*}
 s \geq \omega +\sqrt{\nu \omega^2+\frac{(1+\gamma\varepsilon)^2}{\varepsilon^2}}.
\end{equation*}
Hence the inequality $\|(\mathrm{i}s-A^\top)^{-1}\|<\varepsilon$ will be satisfied if we take
\begin{equation}\label{Smaxineq2}
 S_{\max} = \omega_{\max} +\sqrt{\nu \omega_{\max}^2+\frac{(1+\gamma_{\max}\varepsilon)^2}{\varepsilon^2}}.
\end{equation}
In order to better illustrate the dependence on parameters $\omega_{\max}$ and $\gamma_{\max}$ and the tolerance $\varepsilon$, we can also use the following, slightly worse, bound
\begin{equation*}
    S_{\max} \geq (1+\sqrt{\nu})\omega_{\max}  +\frac{1}{\varepsilon}+\gamma_{\max}.
\end{equation*}

\end{subsection}


\begin{subsection}{Algorithm for the calculation of the finite time horizon $p$--mixed $H_2$ norm}
\label{subsec: Algorithm for norm}

To take into account the oscillating nature of the function we integrate, we divide the integration interval in two parts
\begin{equation*}
2\int_{- S_{\max}}^{0} h_{jk}(t,s) \,\mathrm{d} s =  2\left( \int_{- S_{\max}}^{-S_1} h_{jk}(t,s) \,\mathrm{d} s + \int_{-S_1}^{0} h_{jk}(t,s) \,\mathrm{d} s \right).
\end{equation*}

In the first integral on the right hand side we will use the standard Simpson rule because this integral does not oscillate as highly as the second one. In the second integral on the right hand side we will use the adaptive Simpson rule since the considered function is highly oscillatory on that segment. Moreover, adaptive Simpson is appropriate for our implementation since we use that the number of nodes is a power of 2 and therefore recycling can be done easily.

Algorithm for the calculation of the finite time horizon $p$--mixed $H_2$ norm defined by \eqref{estint} is given by Algorithm \ref{algorithm1}.

\begin{algorithm}[htp]
\caption{Algorithm for the calculation of the finite time horizon $p$--mixed $H_2$ norm}
\label{algorithm1}
\begin{algorithmic}[1]
\REQUIRE
system matrices: $M$ (mass matrix), $K$ (stiffness matrix), $C_{\mathrm{ext}}$ (external damping); system parameters: $\nu$  (determines $C_{\mathrm{int}}$);\\
 tolerances $tol_s$ (tolerance for $S_{\max}$), $tol$ (integration tolerance);\\
the finite time horizon $p$--mixed $H_2$ norm parameters: $p\in [0,1]$ (determine  the target norm),
$r$ (number of undamped frequencies that need to be damped),
$T$ (defines time horizon);\\
$n_t$  (number of nodes for the integration by $t$);\\
$v_1$, $v_2$, \ldots $v_{n_v}$ ($n_v$ considered viscosities);\\
$b_{\max}$ and $b_0$ (maximum and initial number of nodes for adaptive Simpson rule for second integral is $n_2=2^{b_{\max}}$ and $2^{b_0}$);\\
$n_1$ (number of nodes for the first integral, $n_s=n_1+n_2$ is total number of nodes for integration by $s$);\\
$S_1$ (limit for integration by variable $s$ of first integral); 
\ENSURE  Estimation of \eqref{objectiveFun} \\
\textsf{Offline part:}
\STATE\label{smax}{Determine $S_{\max}$ such that equality \eqref{Smaxineq2} holds.
\STATE Determine equidistant nodes $s_1,\ldots,s_{n_1}\in [-S_{\max},-S_1]$, $s_{n_1+1},\ldots,s_{n_s}\in [-S_1,0]$ for integration by s.}
\STATE{Determine nodes $t_1,\ldots,t_{n_t}\in [0,T]$ and weights for integration by $t$.
\STATE \label{first}Compute  $\cos(s_i t_l)$ for every node $s_i$, $i=1,\ldots,n_s$ and $t_l$, $l=1,\ldots,n_t$.}
\STATE\label{Compute matrices GF}{Compute matrices $F(s_i)$ and $G(s_i)$ from \eqref{matrixLs} for all nodes $s_i$, $i=1,\ldots,n_s$.}
\STATE\label{Compute vector fg}{Compute $f(s_i)$ and $g(s_i)$ defined in Proposition \ref{prop:solutionxj} for all nodes $s_i$, $i=1,\ldots,n_s$.}\\
\textsf{Online part:}
\FOR{considered viscosities $v_1$, $v_2$, \ldots $v_{n_v}$}\label{forstep}
\STATE\label{computex}{Compute vectors $x_j^\Im (s_i)$ and $x_j^\Re (s_i)$ from \eqref{xj1}-\eqref{xr2} for all nodes $s_i$, $i=1,\ldots,n_s$ and all $j=1,\ldots,r,n+1,\ldots,n+r$ .}
\FOR{$j=1:r,n+1:n+r$ }
\FOR{$k=1:2n$}
\STATE\label{computeisA}{Evaluate functions $h_{jk}$ on a given grid $(s_i,t_l)$ using formulae from subsection \eqref{subsec:The case of one-dimensional damping} while recycling parts which are time-independent.}
\STATE\label{computeintpos1}{Use the Simpson rule to compute $\int_{- S_{\max}}^{-S_1} h_{jk}(t,s) \,\mathrm{d} s$.}
\STATE\label{computeintpos2}{Use the adaptive Simpson rule to compute $\int_{-S_1}^{0} h_{jk}(t,s) \,\mathrm{d} s $.}
\STATE\label{computeintpot}{Use the  quadrature rule to compute $\int_0^T  \left( \int_{- S_{\max}}^{ S_{\max}} h_{jk}(t,s) \,\mathrm{d} s \right)^2  \,\mathrm{d}t$. }
\ENDFOR
\ENDFOR
\ENDFOR
\end{algorithmic}
\end{algorithm}

\end{subsection}

\begin{subsubsection}{Algorithm description}

The Algorithm \ref{algorithm1} consists of two parts; the offline and the online part.

First, we give more details regarding the \textbf{offline part.} In the offline part, we use the fact that some intermediate calculations can be effectively stored in matrices that do not depend on viscosity, and therefore, they can be calculated only once. In particular, in the offline part, we prepare the matrix of the type $ \mathbb{R}^{n_s \times n_t}$ which entries are $\cos(s_i t_l)$, $i=1,\ldots,n_s$, $l=1,\ldots,n_t$.
Also, we prepare matrices of the type $  \mathbb{R}^{n_s \times n}$ with rows being the diagonal elements of matrices $F(s_i)$ and $G(s_i)$, $i=1,\ldots,n_s$. As entries of these matrices do not depend on the viscosities they can be calculated in the offline part.

In the \textbf{online part} we organize calculations in such a way that we can recycle computationally demanding parts. First, we define tensors $x_{\Im}$, $x_{\Re} \in \mathbb{R}^{n_s \times 2r \times n}$, where $x_{\Im}(s_i,j,:)= x_j^\Im (s_i)$ and $x_{\Re}(s_i,j,:)=x_j^\Re (s_i)$, for $i=1,\ldots,n_s$ and $j=1,\ldots,r,n+1,\ldots,n+r$. Here $x_j^\Im (s_i)$ and $x_j^\Re (s_i)$ are defined in \eqref{xj1} - \eqref{xr2}.

Then, in the most computationally demanding part of the algorithm line \ref{computeisA} - line \ref{computeintpot}, for every $k=1,\ldots,2n$ and $j=1, \ldots, r,n+1, \ldots, n+r$, we calculate the terms $h_{jk}(t_l,s_i)$ (whose formulae are derived in subsection \ref{subsec:The case of one-dimensional damping}) and we also use recycling for the parts that do not depend on time variable.

For the computation of the integral in line \ref{computeintpos1} 
we use the standard Simpson rule, hence we calculate the terms $h_{jk}(t_l,s_i)$, $i=1,\ldots,n_{s_1}$, $l=1,\ldots,n_t$.




For the 
integral in line \ref{computeintpos2} we start from the initial mesh of integration nodes (initially, we have $2^{b_{0}}$ nodes). Here, we are using the adaptive Simpson approach, which means that we pick iteratively denser meshes until we reach the prescribed tolerance $tol$. In this process, we recycle previously calculated function values $h_{jk}(t,s)$ to accelerate computations. Moreover, when the difference between the nodes reaches the maximal number of segment subdivisions (determined by the parameter $2^{b_{\max}}$), the current approximation on the segment is accepted.
We emphasize that in the steps \ref{computeisA}, \ref{computeintpos1} and \ref{computeintpos2} we benefit greatly from recycling. Recycling of the data from the offline part is made easier by the use of the adaptive approach, as we use the equidistant mesh with the number of nodes of the form  $2^l$.

\begin{remark}
The algorithms have been implemented in Julia (see \cite{julia}). Julia low level programming enables efficient implementation comparable with standard BLAS routines. In Julia  we were able to efficiently implement standard and adaptive Simpson quadratures, including nested loops with simple operations. Of course, efficiency strongly depends on the number of nodes $t_j$ and $s_i$ and on how much recycling is used. Therefore, in the offline  phase, we have prepared data needed for the calculation of the target value. Then we have used adaptive Simpson  approach for the calculation of the integral over the variable $s$. The main reason for that is that it allows an implementation of adaptive quadrature  that uses equidistant nodes, while on the same time we can recycle data obtained from previous steps as well as the data prepared in the offline part.
\end{remark}
\end{subsubsection}

\begin{subsubsection}{Algorithm parallelization}

The most demanding parts of our algorithm can be parallelized and therefore our approach can be additionally accelerated. Here we would like to emphasize where we have used Julia's multithreading environment. In particular, we have used the macro \verb|threads|. First, in the offline phase we have used the macro \verb|threads| in line \ref{first} and line \ref{Compute matrices GF} since this part includes generation of matrices with rows which are diagonal elements of matrices $F(s_i)$ and $G(s_i)$ for all nodes $s_i$ (that do not depend on the time nodes), while we also form a matrix that stores values $\cos(s_i t_l)$, for all nodes $s_i$, $i=1, \ldots,n_s$ and all nodes $t_l$, $l=1, \ldots,n_t$, which depends on the time nodes as well.

Furthermore, the main benefit from parallelization comes from acceleration of the online part. In particular, we have used the macro \verb|threads| in the inner loop for calculating the tensors $x_{\Im}$ and $x_{\Re}$. Also, the inner loop over $k$ depends on the dimension $n$, so lines \ref{computeisA} - \ref{computeintpot} have been calculated using the macro \verb|threads|.

We would like to emphasize the benefit from using Algorithm \ref{algorithm1} compared to the approach that calculates Lyapunov equation given by \eqref{directFormula}.
In the next section, we will calculate the number of  floating point operations needed for one evaluation of $p$–mixed $H_2$ norm for both approaches.

\end{subsubsection}


\begin{subsection}{Complexity analysis}

An alternative approach for the calculation of the finite time horizon $p$--mixed $H_2$ norm uses the formula \ref{directFormula} and we will call it  a Lyapunov based approach. This approach  requires calculation of the Lyapunov equation and calculation of the matrix exponential. This means that Lyapunov based approach needs $n_v\mathcal{O}(n^3)$ floating point operations for calculation of $p$–mixed $H_2$ norm for $n_v$ different viscosities.

This Lyapunov based approach can be accelerated by using model order reduction techniques as we mentioned in the introduction, but also it can be accelerated by using the low rank structure that appears in our case. In particular, the objective function can be calculated using an approach that is based on function calculation using low-rank updates; see, e.g., \cite{BeckermannCortinovisKressner21, BeckKressnerSch18}. This can be applied by calculating the matrix exponential function and by solving the Lyapunov equation given by \ref{directFormula}. Comparison with this approach is not given since we do not have relevant implementation of this approach. Moreover, we would like to emphasize that when the finite horizon $T$ is changed, the Lyapunov equation \ref{directFormula} does not need to be solved again, but the matrix exponential function needs to be calculated repeatedly. On the other hand, our approach includes formulae constructed in such a way so that the most expensive part (from the calculation point of view) can be recycled. Below we give more details on the complexity of our approach.

In the analysis of the complexity of Algorithm \ref{algorithm1}, we will separately study the online and the offline part. Moreover, there are two levels of offline parts. One includes certain complexity that does not depend even on the external damping and therefore can be done only once for all viscosities, damping positions and different time horizons. Main cost within this part comes from the simultaneous diagonalization defined by \eqref{MKdiag} that requires  $\mathcal{O}(n^3)$ (floating point) operations, but only once. Note that the simultaneous diagonalization is always necessary if the internal damping is defined in terms of the critical damping in \eqref{C_u}.
Then, once the damping positions are fixed, we can evaluate Algorithm \eqref{algorithm1} as follows
\begin{itemize}
\item Offline
  \subitem  preparation of offline data includes line \ref{smax} - line \ref{Compute vector fg} that requires: $n_s\mathcal{O}(n )$ operations
  \item Online, for each viscosity line \ref{computex} - line \ref{computeintpot} include
  \subitem line \ref{computex} that requires: $n_s\mathcal{O}(rn)$ operations
  \subitem for all $j$ and $k$ we have line \ref{computeisA} - line \ref{computeintpot} where
  \subsubitem  line \ref{computeisA} requires: $ \mathcal{O}(n_sn) $ operations
  \subsubitem  line \ref{computeintpos1} requires: $\mathcal{O}(n_t n_1)$ operations
  \subsubitem  line \ref{computeintpos2} requires: $\mathcal{O}(n_t n_2)$ operations
  \subsubitem  line \ref{computeintpot} requires: $\mathcal{O}(n_t)$ operations
\end{itemize}
This means that line \ref{computeisA} - line \ref{computeintpot} require  $ n_s\mathcal{O}(rn)+ n_tn_s\mathcal{O}(rn)+ n_s\mathcal{O}(rn^2)$ floating point operations.

Taking all into account, we can conclude that $p$–mixed $H_2$ norm for $n_v$ different viscosities can be calculated  by Algorithm \ref{algorithm1} using $n_vn_s\mathcal{O}(rn^2)$  floating point operations excluding the possible computational cost of the simultaneous diagonalization \eqref{MKdiag}, which is independent on the choice of all parameters except matrices $M$ and $K$. In particular, it does not depend on the parameters $T$, $p$, and the choice of the external damping matrix $C_{\mathrm{ext}}$ and so the cost can be taken as negligible in the framework of optimization or online simulation. Moreover, our estimation on the required number of floating point operations depends on $n_s$ which does not depend directly on $n$.

Here we would like to emphasize that in our approach when parameter $T$, which defines the time horizon, is changed we can use recycling on that level too. In particular, line \ref{computeisA} is the most demanding step in the algorithm but it does not depend on the time nodes. Moreover, when $T$ is slightly changed we need to evaluate Algorithm \ref{algorithm1} for very small $n_t$ while all the data that do not depend on time nodes can be recycled. From the complexity perspective this means that the floating point operations from line \ref{computex} and line \ref{computeisA} now belong to the offline part. Taking all this into account we obtain that calculation of $p$–mixed $H_2$ norm for $n_T$ different time horizons can be calculated  by Algorithm \ref{algorithm1} using
\begin{displaymath}n_s\mathcal{O}(rn)+   n_s\mathcal{O}(rn^2)+ n_T n_tn_s\mathcal{O}(rn) \end{displaymath}
floating point operations. In contrast, Lyapunov based approach needs $n_T\mathcal{O}(n^3)$ floating point operations for the calculation of $p$–mixed $H_2$ norm for $n_T$ different time horizons.

This analysis  shows that we can efficiently analyze the influence of different final times $T$ on the $p$--mixed $H_2$ norm. Therefore, one approach for practical determination of a good final time $T$ could be based on the determination of $T$ for which the $p$--mixed $H_2$ norm stagnates. In the next section, we will illustrate this in an numerical example.

\end{subsection}


\end{section}

\begin{section}{Numerical experiments}
\label{sec: Numerical experiments}
 
In this section, we present numerical examples in order to illustrate the behaviour of the $p$-mixed $H_2$ norm for different choices of the finite time horizon $T$ and advantages of our approach compared with the Lyapunov based approach. Computations have been carried out on a workstation with 64-bit Linux operating system and with an AMD\textregistered Ryzen Threadripper\texttrademark processor with 64 CPUs, 128 threads and 256 GB DDR4 RAM. Moreover, for the sake of time comparison on a standard computer, in Example \ref{example2}, computations have also been tested on a laptop with Intel\textregistered Core\texttrademark  i7-9750H processor with 6 CPUs, 12 threads, 8 GB RAM and  64-bit version of Windows. Numerical experiments are performed using Julia \cite{julia}, on the workstation we have used  Version 1.6.3  with 32 threads, while on the laptop we have used Version 1.6.0 with 6 threads.

\begin{example} \label{example1}
We consider an $n$-mass oscillator or oscillator ladder with one damp\-er, shown in Figure \ref{n-massOscillator}, which
describes the mechanical system of $n$ masses and $n+1$ springs. Similar models were considered e.g. in the papers \cite{BennerTomljTruh11}, \cite{NakTomTr19}, \cite{TrTomljPuv18} and the book \cite{VES2011}.

 In this example, we are interested in analyzing the external damping that significantly influences the system. To this end, we are considering effective viscosity, i.e.\ the threshold value after which the finite time horizon $p$--mixed $H_2$ norm drops significantly. We noticed through numerical experiments that such a value exists for all systems we have considered. Since in this example we take small $n$, we use the Lyapunov based approach.

\begin{figure}[htbp]
\centering
\includegraphics[width=10cm]{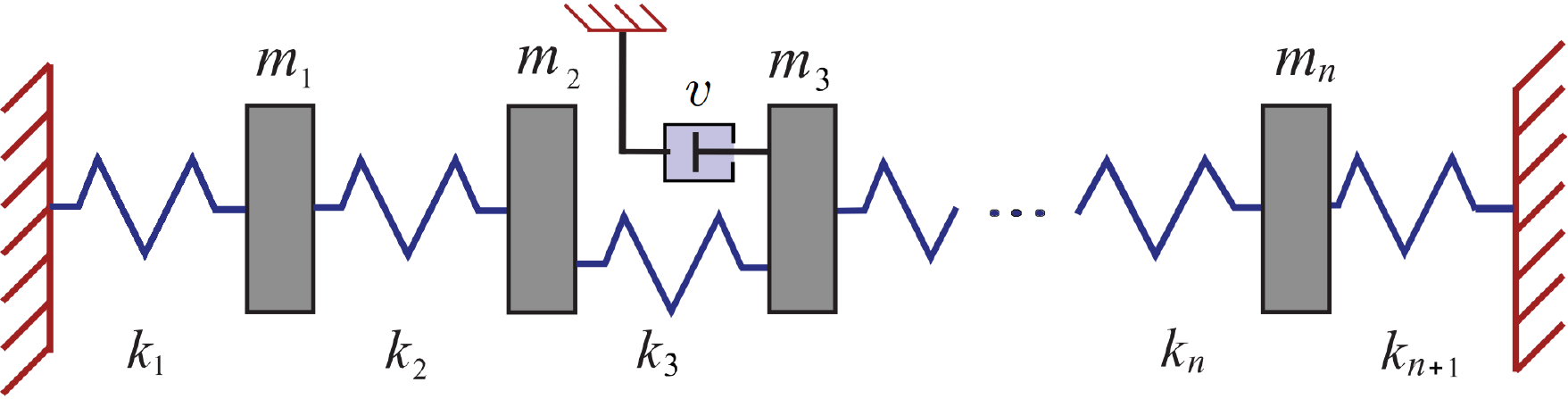}
 \caption{The $n$-mass oscillator with one grounded damper}
 \label{n-massOscillator}
\end{figure}
For such a mechanical system the mathematical model is given by
\eqref{MDK}, where the mass and stiffness matrices are
\begin{align*}
M&=\diag(m_1,m_2,\ldots,m_n),\\
K&=\left(%
\begin{array}{ccccc}
  k_1+k_2 & -k_2 &  &  &  \\
  -k_2    & k_2+k_3& -k_3 &  &  \\
          & \ddots   & \ddots & \ddots &  \\
          &      & -k_{n-1} & k_{n-1}+k_{n} & -k_n   \\
          &      &  & -k_n & k_n+k_{n+1} \\
\end{array}
\right).
\end{align*}

We choose matrices  $B_2, E_1$ and $  E_2$ such that \eqref{matrices EZ} holds. Such matrices can be determined directly from our system matrices and more details can be found in \cite[Section 4]{NakTomTr19}.

We will consider the following configuration
\begin{align*}
 n &= 200; \quad k_j = \frac{n}{2}, \quad \forall j; \quad
   m_j   =  \begin{cases} \frac{n - 2 j}{10}, \quad & j=1,\,\ldots,\frac{n}{4}, \\
      \frac{\frac{n}{4}+j}{10} , \quad & j=\frac{n}{4}+1,\,\ldots,\,n.
                 \end{cases}
\end{align*}

We will consider one damper, but the damping position will be changed, so the
external damping is defined by \eqref{Cext} and  the internal damping $C_{\mathrm{int}}$ is
defined as in \eqref{C_u}, with $\alpha=0.005$. In this example for the illustration and comparison purposes we will consider four different damping positions, that is, we will consider that vector $e$ from \eqref{Cext}, that encodes damping positions, is equal to $e_i$ where  $i=10, 80,110, 160$.

We consider  damping of $1\,\%$ smallest eigenfrequencies of the system which means that we have $r=\frac{\displaystyle n}{\displaystyle 100}=2$ and parameter $p$ that defines $p$--mixed $H_2$ norm is taken to be $0.5$.

In this example we will show the influence of the parameter $T$ in \eqref{objectiveFun}, that is the influence of the time horizon in the $p$--mixed $H_2$ norm.  This influence is shown on \eqref{optviscfigure}.

\begin{figure}[htbp]
\centering
\includegraphics[width=12cm]{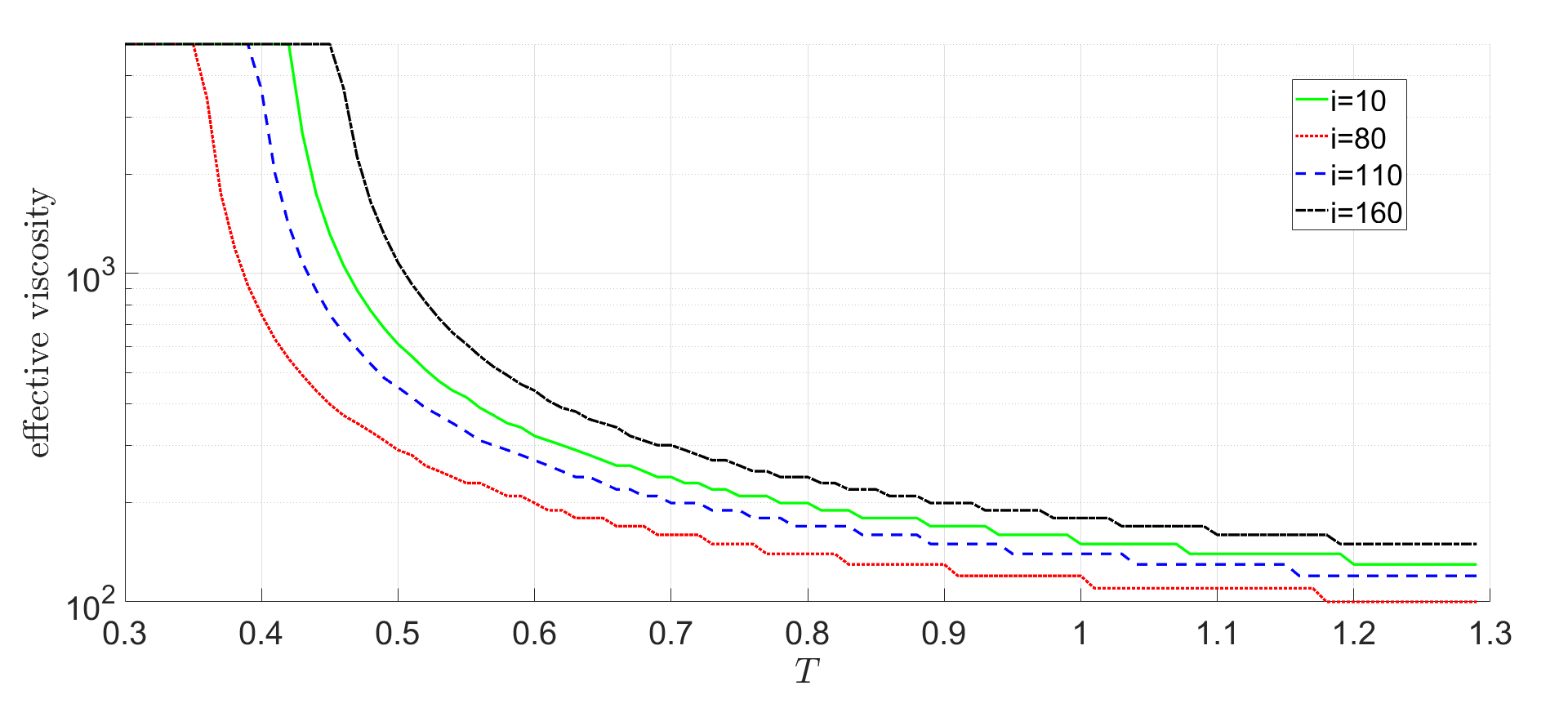}
  \caption{(Example \ref{example1}) The influence of the integration time on the magnitude of the effective viscosity for dimension $n=200$ and for four different damping positions. By effective viscosity we mean the threshold value after which the finite time horizon $p$--mixed $H_2$ norm drops significantly. The computation is done using the workstation.}
 \label{optviscfigure}
\end{figure}

As is to be expected, for very small times $T$ it is hard to achieve significant damping effects, therefore a very large viscosity is needed to significantly reduce the finite time horizon $p$--mixed $H_2$ norm, which is usually physically infeasible. On the other hand, when $T$ is increased we observe major decay in effective viscosities for all considered damping positions. This means that for moderate times $T$ effective viscosities vary within the appropriate values, and for large times $T$ the curve is close to the case $T=\infty$. Of course, it is hard to state what do we mean by moderate $T$, but exactly this decay gives us this information. In this example, from our analysis we can observe that relevant time horizon with  reasonable effective viscosity  starts at around $1$.
\end{example}

\begin{example}\label{example2}
In this example we will have the same configuration as in the previous example, but with the dimension $n=2000$. We  have calculated an approximation of the finite time horizon $p$--mixed $H_2$ norm by using Algorithm \ref{algorithm1} with the following initial requirements:
\begin{align*}
tol&=10^{-5},  &n_t&=20,\\
 n_v&=20,  & n_1&=599,\\
  S_1&=\omega_n/25,  & b_0&=8,\\
  b_{\max}&=12,  & tol_s&=0.05,
\end{align*}
and we consider the following viscosities  $v_1  =75 ,  v_2=150, \dots ,v_{20}=1500$.

As in the previous example, we have used $p = 0.5$ and we would like to damp $1\,\%$ of the undamped eigenfrequencies, which means that $r=20$, while internal damping is determined using $\alpha=0.005$.
Here, from the similar numerical analysis as in the previous example, we can conclude that $T$ should be larger than $1$, and therefore in the continuation of this example  we illustrate the efficiency of our approach for $T=2$. Moreover, the influence of $T$ is significant and with our approach we can efficiently calculate the finite time horizon $p$--mixed $H_2$ norm for several values of $T$, since the major computational cost taken care of in the offline part.

To present comparison with a Lyapunov based approach, we will calculate the finite time horizon $p$--mixed $H_2$ norm with new approach and compare it with the Lyapunov based approach that uses the formula \eqref{directFormula}.
 Here Lyapunov based approach is implemented in such a way that we solve Lyapunov equation and matrix exponential directly. In particular, in Julia the matrix exponential is calculated  using one of the most widely used method based on \cite{Higham05}, while the algorithm for solving the Lyapunov equation is based on LAPACK routines that uses direct solvers for Lyapunov equation. For  more details see, e.g. \cite{DATTA04,ANT05,moler2003ndw,Golub2013,lapack99}. Therefore we can use this value as the exact solution.
  Figure \ref{meanerr} presents the average relative error for these two approaches for $20$ equidistant viscosities, from $75$ to $1500$, and shows the average relative error for all considered damping positions $e_i$, where  $i=100, 200, \ldots, 1900$.
Then, for four different damping positions, that is for $e=e_i$, where $i=200,800,1100,1600$, Figure \ref{relerr} shows the relative error for these two approaches for all $20$ equidistant viscosities.  We can see that our approach for given tolerances results with satisfactory accuracy.

\begin{figure}[htbp]
\centering
\includegraphics[width=12cm]{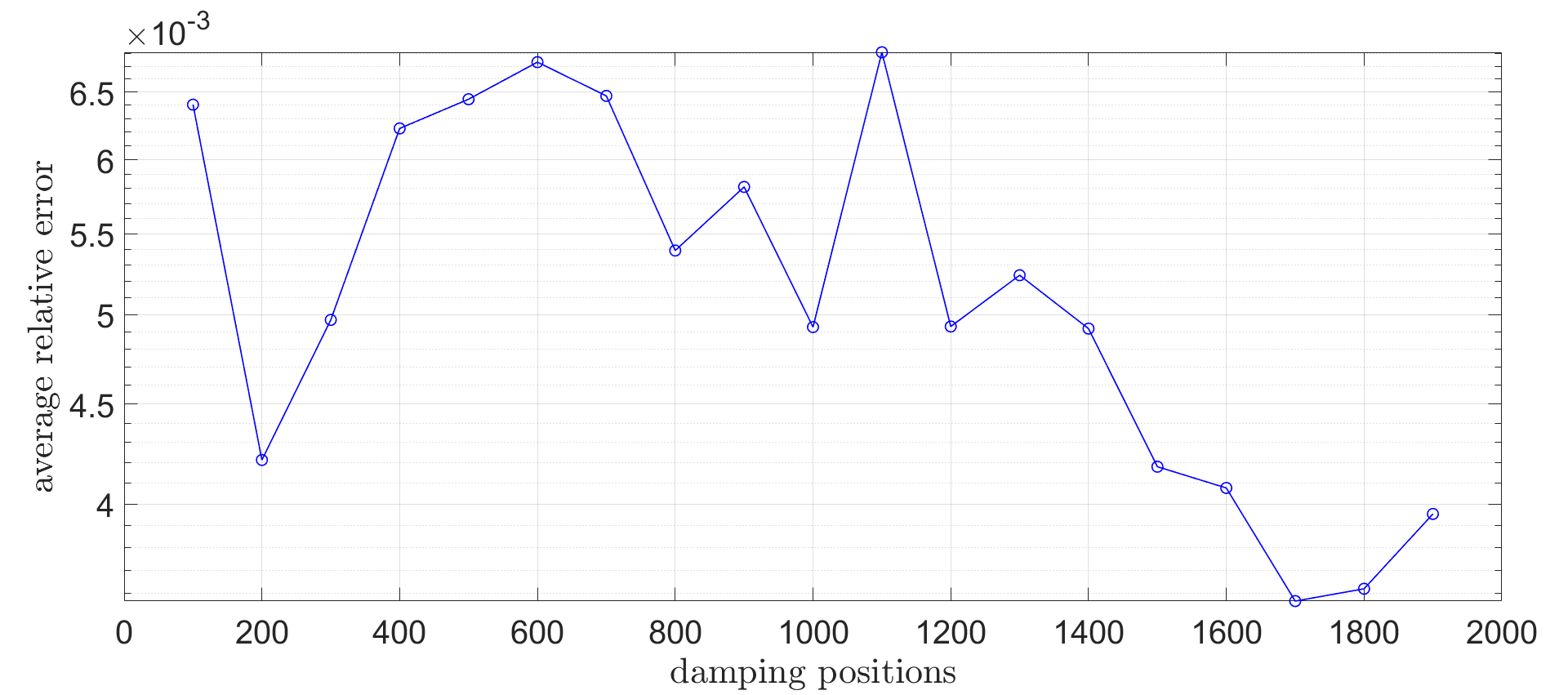}
\caption{(Example \ref{example2}) The average relative error for $20$ equidistant viscosities from $75$ to $1500$, for the calculation of the finite time horizon $p$--mixed $H_2$ norm using Algorithm \ref{algorithm1}, compared with the Lyapunov based approach, for 19 different damping positions. The computation is done using the workstation.}
\label{meanerr}
\end{figure}

\begin{figure}[htbp]
\centering
\includegraphics[width=12cm]{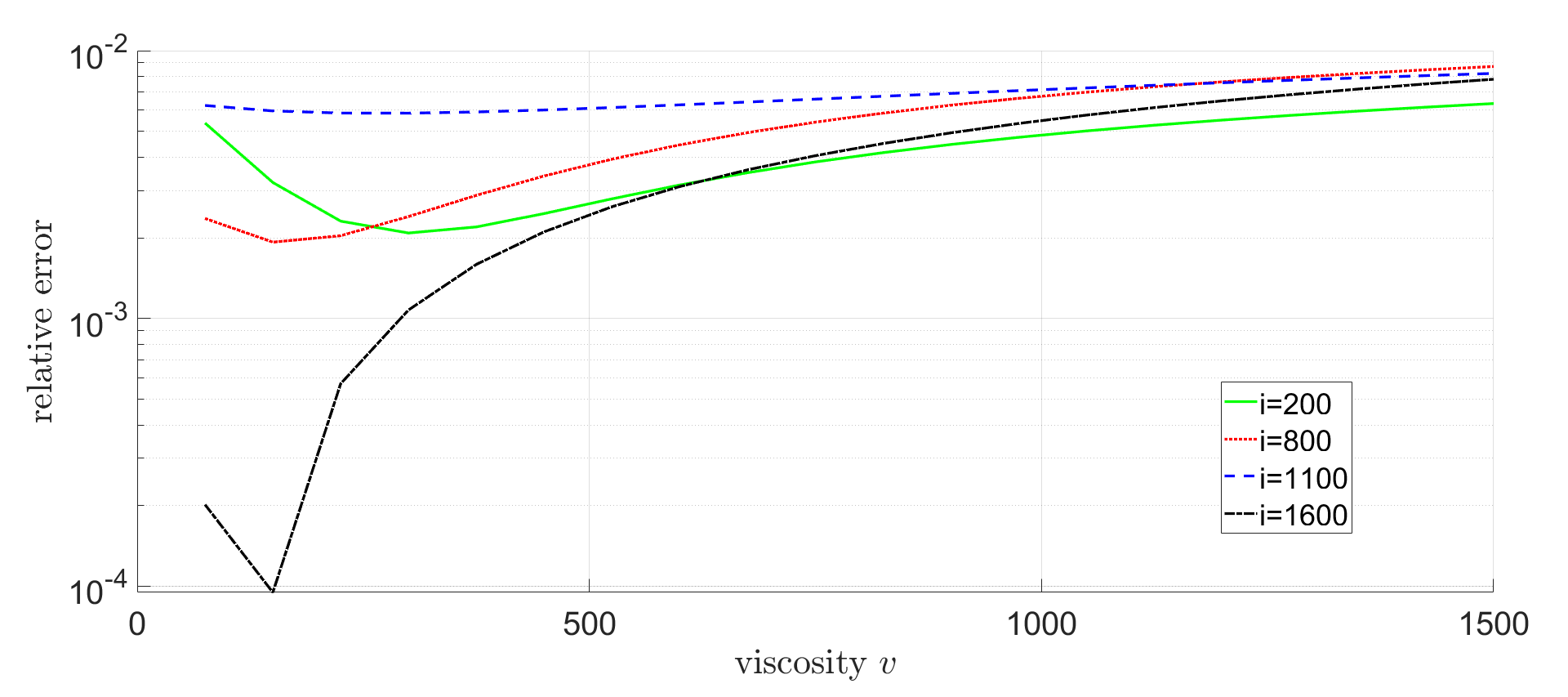}
\caption{(Example \ref{example2}) The  relative error for the calculation of the finite time horizon $p$--mixed $H_2$ norm using Algorithm \ref{algorithm1}, compared with the Lyapunov based approach, for four different damping positions and $20$ equidistant viscosities from $75$ to $1500$. The computation is done using the workstation.}
\label{relerr}
\end{figure}

For a time comparison, first we have used a laptop to calculate the finite time horizon $p$--mixed $H_2$ norm, using Algorithm \ref{algorithm1} and using Lyapunov based approach. For four different damping positions given in Figure \ref{relerr}, average acceleration factor is $2.5$.

\begin{figure}[htbp]
\centering
\includegraphics[width=12cm]{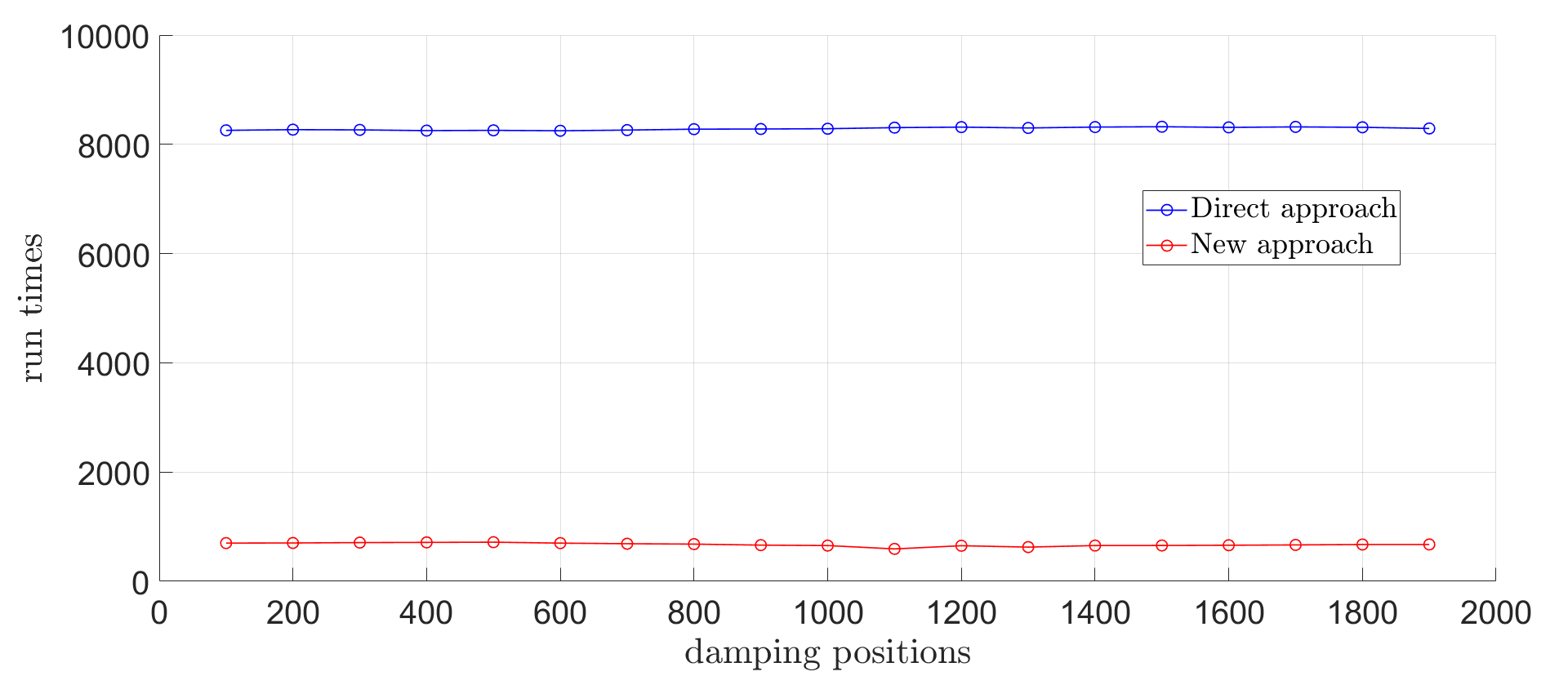}
\caption{(Example \ref{example2}) The time required for the calculation of the finite time horizon $p$--mixed $H_2$ norm using Algorithm \ref{algorithm1} compared to the time required for the Lyapunov based approach, for 19 different damping positions. The computation is done using the workstation.}
\label{accfact}
\end{figure}

This can be improved by the efficient usage of tensor structures that arise in Algorithm \ref{algorithm1} which is illustrated in Figure \ref{accfact}. In particular, Figure \ref{accfact} presents the time required for the calculation of the finite time horizon $p$--mixed $H_2$ norm using Algorithm \ref{algorithm1} and the time required for the Lyapunov based approach.

Current implementation uses the benefit of large number of threads available on the workstation. Thus, it is optimized for better usage of multithreading environment, and this is also confirmed in numerical tests. In particular, on the workstation  an average acceleration factor for four considered damping positions presented on Figure \ref{relerr} is $12.2$.

\end{example}

\begin{example} \label{example3}
In this example, we consider the mechanical system shown in Figure \ref{titrajni_sistem3rows}, consisting of three rows of $d$ masses and $d+1$ springs which are, on the left-hand side, connected to the fixed base. Springs in each row have the same stiffness equal to $k_1$, $k_2$ and $k_3$. On the right-hand side, they are connected to one additional mass, which is connected to the fixed base with a spring of stiffness $k_4$.
Similar models were considered in the paper \cite{BennerTomljTruh10}.

\begin{figure}[htbp]
\centering
\includegraphics[width=8cm]{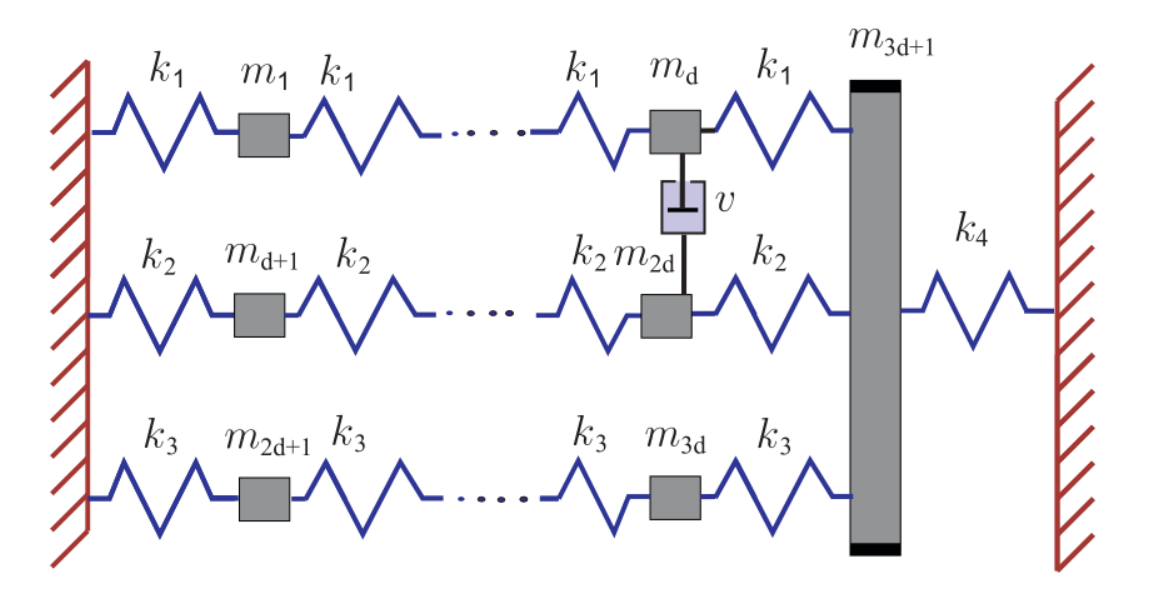}
 \caption{$(3d+1)$-mass oscillator with one damper}
 \label{titrajni_sistem3rows}
\end{figure}
For such a mechanical system the mathematical model is given by
\eqref{MDK}, where the mass and stiffness matrices are
\begin{align*}
M&=\diag(m_1,m_2,\ldots,m_n),\\
K&=\left(%
\begin{array}{ccccc}
   K_1 &  &  & -\kappa_1\\
    & K_2 &  & -\kappa_2\\
    &  & K_3 & -\kappa_3\\
   -\kappa_1 & -\kappa_2 & -\kappa_3 & k_1+k_2+k_3+k_4
\end{array}
\right),\quad \mbox{with}\\
K_i&=k_i \left(%
\begin{array}{ccccc}
  2 & -1 &  &  &  \\
  -1    & 2 & -1 &  &  \\
          & \ddots   & \ddots & \ddots &  \\
          &      & -1 & 2 & -1   \\
          &      &  & -1 & 2 \\
\end{array}
\right),
\kappa_i =\left(%
\begin{array}{ccccc}
  0  \\
  0  \\
  \vdots  \\
  0   \\
  k_i \\
\end{array}\right), i=1,2,3.
\end{align*}
We have the following configuration
\begin{align*}
 d = 800; \quad k_j=\begin{cases} 800,\quad j=1,\\ 600,\quad j=2,\\ 700,\quad j=3;
 \end{cases}
   m_j   =  \begin{cases} 1000-2j, \quad & j=1,\,\ldots,\frac{d}{2}, \\
      j-200, \quad & j=\frac{d}{2}+1,\,\ldots,\,d,\\
      j+100, \quad & j=d+1,\,\ldots,\,2d,\\
      n-2j, \quad & j=2d+1,\,\ldots,\,3d,\\
      2000, \quad & j=n=3d+1.
    \end{cases}
\end{align*}
We will consider here one damper with a different damping geometry compared to the previous example. Damper will be located between two masses in  different rows and we will consider six different damping positions. The corresponding vectors $e$ from \eqref{Cext}, denoted by $e^i$ are given by \begin{align}
    (e^i)_j= \begin{cases} 1,\quad &j=i\\
    -1, \quad & j=i+d\\
    0, \quad & \text{otherwise}
\end{cases}, \quad j=1,\ldots ,n, \label{dampposThrdEx}
\end{align} for  $i=20, 320,620, 920,1220,1520$.
The internal damping $C_{\mathrm{int}}$ is
defined as in \eqref{C_u}, with $\alpha=0.005$.

As in previous examples, we consider  damping of approximately $1\,\%$ smallest eigenfrequencies of the system, that is $r=24$, and parameter $p$ that defines $p$--mixed $H_2$ norm is taken to be $0.5$. We  have calculated an approximation of the finite time horizon $p$--mixed $H_2$ norm by using Algorithm \ref{algorithm1} with the following initial requirements:
\begin{align*}
tol&=10^{-7},  &n_t&=20,\\
 n_v&=20,  & n_1&=799,\\
  S_1&=\omega_n/15,  & b_0&=10,\\
  b_{\max}&=14,  & tol_s&=0.05,
\end{align*}
and we consider the following viscosities:  $v_1  =10 ,  v_2=210, v_3=410, \dots ,v_{20}=3810$.
We illustrate the efficiency of our approach for $T=15$.

To present comparison with a Lyapunov based approach, as in the previous example, we will calculate the finite time horizon $p$--mixed $H_2$ norm with the new approach and compare it with the Lyapunov based approach based on the formula \eqref{directFormula}.
For all six damping positions defined by \eqref{dampposThrdEx}, Figure \ref{ex32} shows the relative error for these two approaches for all $20$ viscosities, from $10$ to $3810$.  We can see that our approach for given tolerances results with satisfactory accuracy.


\begin{figure}[htbp]
\centering
\includegraphics[width=12cm]{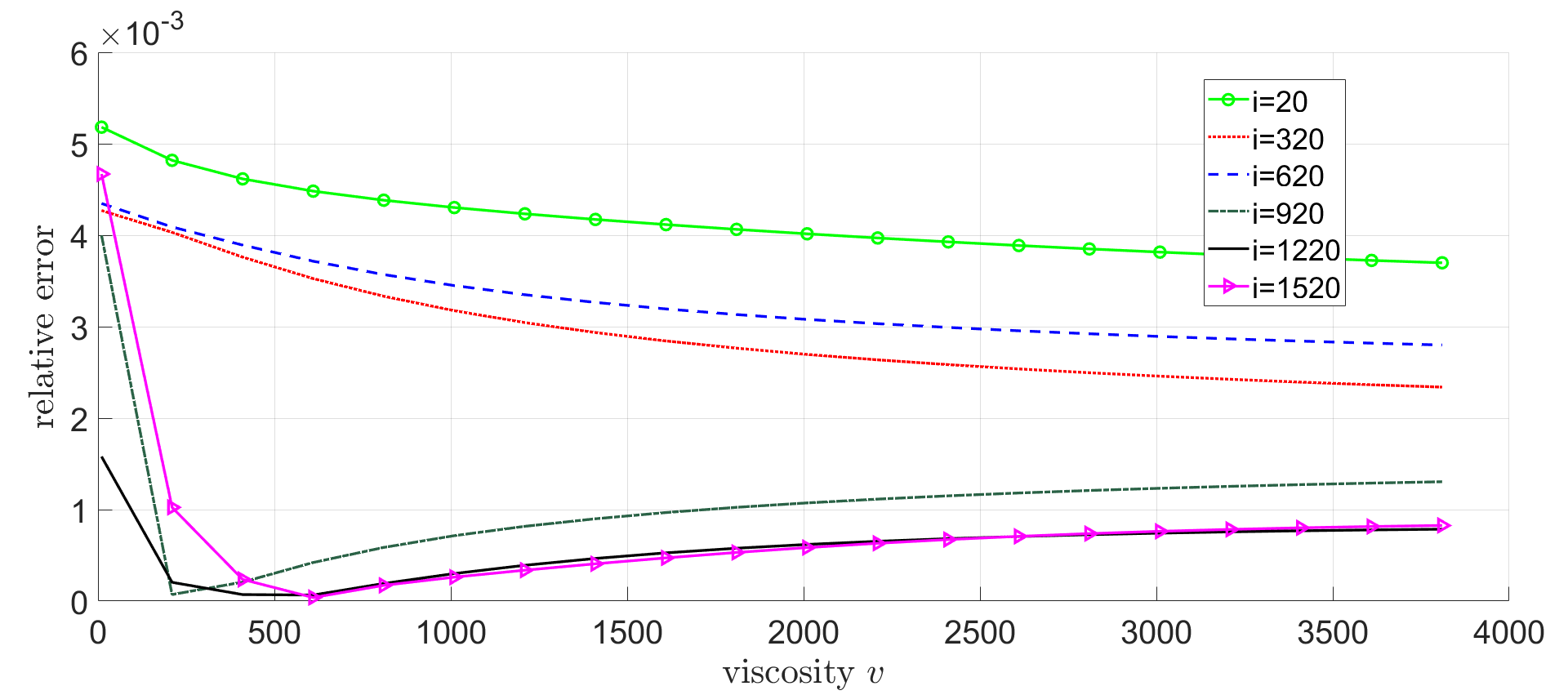}
\caption{(Example \ref{example3}) The  relative error for the calculation of the finite time horizon $p$--mixed $H_2$ norm using Algorithm \ref{algorithm1}, compared with the Lyapunov based approach, for six different damping positions and $20$ equidistant viscosities between $10$ and $3810$. The computation is done using the workstation.}
\label{ex32}
\end{figure}

For a time comparison, Figure \ref{ex33} presents the time required for calculating the finite time horizon $p$--mixed $H_2$ norm using Algorithm \ref{algorithm1} and the time required for the Lyapunov based approach.

\begin{figure}[htbp]
\centering
\includegraphics[width=12cm]{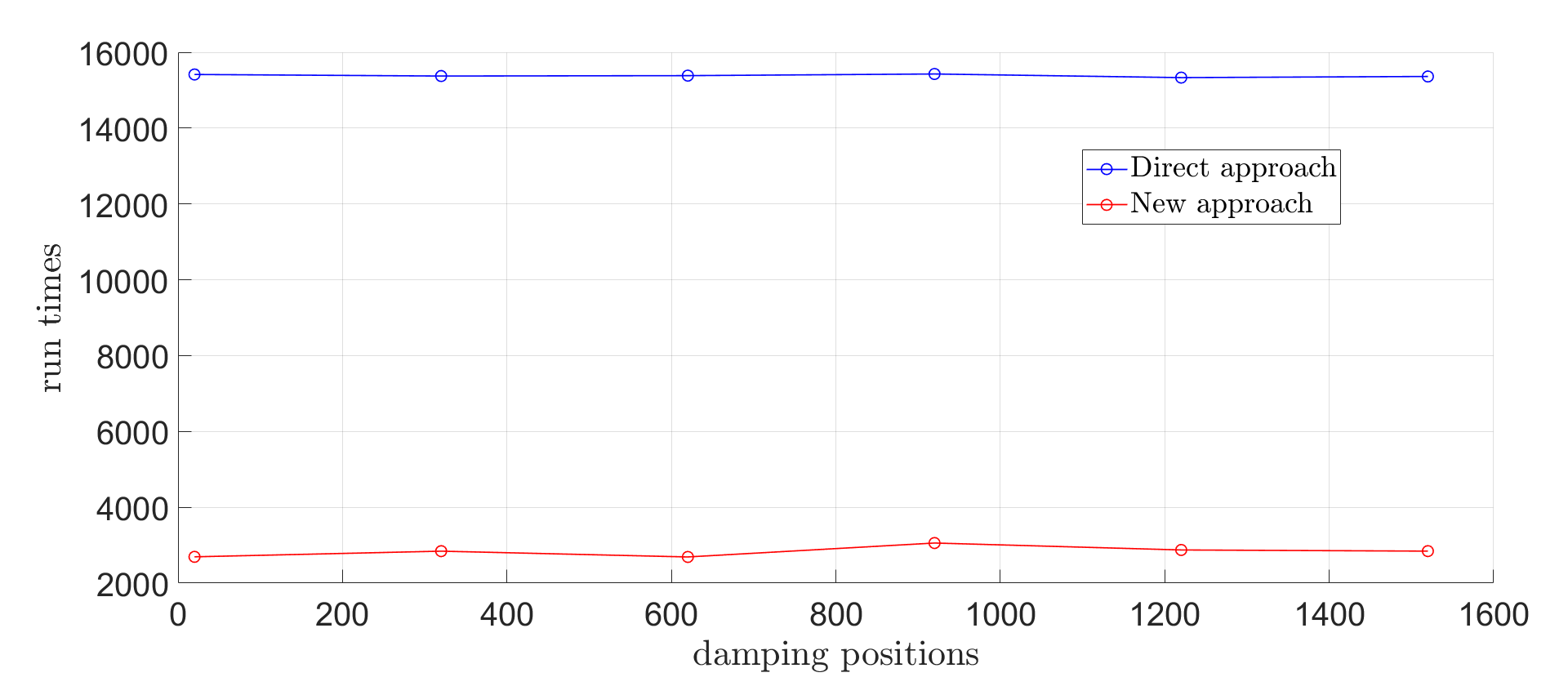}
\caption{(Example \ref{example3}) The time required for calculating the finite time horizon $p$--mixed $H_2$ norm using Algorithm \ref{algorithm1} compared to the time required for the Lyapunov based approach, for six different damping positions. The computation is done using the workstation.}
\label{ex33}
\end{figure}

Here we obtain that an average acceleration factor for six considered damping positions presented on Figure \ref{ex32} is $5.4$.


\end{example}

A limitation of the proposed method is that with this method, one can not efficiently reach very high accuracies. More accurate approximations may be significantly costlier, which makes this method not feasible for specific applications where high accuracy is needed. However, regarding typical vibrational systems, it is unrealistic to expect that the viscosity needs to be calculated with an accuracy greater than $10^{-3}$, as typical damping devices can not be that precisely calibrated. Thus, with the proposed method, we can obtain satisfactory accuracy for damping purposes with a significant time speed-up with respect to the standard approaches.

\end{section}

\bibliographystyle{plain}

\begin{section}{Conclusions}
We considered a control problem for damped vibrational systems where the performance measure of the system is chosen to be $p$-mixed $H_2$ norm over the finite time horizon. The algorithm presented in the paper offers an efficient calculation of this norm in the case when the system is dependent on one parameter and the number of inputs or outputs of the system is significantly smaller than the order of the system. This approach can be extended to the case of the multi-parameter setting as well as for other parameter-dependent systems where an efficient calculation of finite time horizon $H_2$ norm is needed. In future work we will extend this approach also to the optimal control problems for damped vibrational systems and also
other problems in control theory.
\end{section}

\appendix
\begin{section}{Proofs of Propositions \ref{systems}, \ref{prop:solutionxj} and \ref{prop:solutionHjk}}
\label{Appendix}

\begin{proof}[Proof of Proposition  \ref{systems}]
For matrix $L(S)$, we have
\begin{align*}
    L(s) &= \diag (l_1(s),\ldots, l_n(s)) \\
    &= \diag \left( \frac{1}{\omega_1^2 - s^2 + \mathrm{i} s \nu \omega_1 }, \ldots, \frac{1}{\omega_n^2 - s^2 + \mathrm{i} s \nu \omega_n }\right) = F(s) - \mathrm{i} G(s).
\end{align*}
For $ 1\le j \le r$  we have
\begin{equation*}
	\begin{bmatrix}
	I + s D G(s) & - sD F(s) \\
	sD F(s) & I + s D G(s)
	\end{bmatrix}
	\begin{bmatrix}
	x_j^{\Re} \\ x_j^{\Im}
	\end{bmatrix}
	=   \begin{bmatrix}
	  D \Omega F(s) e_j \\ - D \Omega G(s) e_j
	\end{bmatrix}.
\end{equation*}
For $ n+1\le j \le n+r$  we have
\begin{equation*}
	\begin{bmatrix}
	I + s D G(s) & - sD F(s) \\
	sD F(s) & I + s D G(s)
	\end{bmatrix}
	\begin{bmatrix}
	x_j^{\Re} \\ x_j^{\Im}
	\end{bmatrix}
	=  \begin{bmatrix}
	 e_{j-n} \\ 0
	\end{bmatrix}.
\end{equation*}
It is easy to see from the construction of $G(s)$ and the assumption on $D$ that the matrix $I + s D G(s)$ is always non-singular. If we multiply this systems from the left by the matrix
\begin{displaymath}
\begin{bmatrix}
\left( I + s D G(s) \right)^{-1} & 0 \\ -  sD F(s) \left( I + s D G(s) \right)^{-1} & I
\end{bmatrix},
\end{displaymath}
we obtain the stated result.
\end{proof}

\begin{proof}[Proof of Proposition \ref{prop:solutionxj}]
First, by using the Sherman–Morrison–Woodbury formula (see, e.g., \cite{Golub2013}) we obtain
\begin{displaymath}
\left( I + s D G(s) \right)^{-1} = I - s U \left( \frac{1}{\gamma}+ s U^\top G(s) U \right)^{-1} U^\top G(s) .
\end{displaymath}
Note that $U^\top G(s)U\in\mathbb{R}$
so the inverse on the right hand side is trivial to solve,
\begin{displaymath}
\left( I + s D G(s) \right)^{-1} = I - \frac{s\gamma}{1+s\gamma g}UU^\top G(s).
\end{displaymath}
With this, \eqref{eq:sustav1} and \eqref{eq:sustav2} become
\begin{multline*}
    \begin{bmatrix}
	I & -\frac{\displaystyle s\gamma}{\displaystyle 1+s\gamma g(s)}U U^\top F(s)
	\\ 0 & s^2 \gamma^2 U U^\top F(s)(I-\frac{\displaystyle s\gamma}{\displaystyle 1+s\gamma g(s)}U U^\top G(s)) U U^\top F(s) + I +s\gamma U U^\top G(s)
	\end{bmatrix}
	\begin{bmatrix}
	x_j^{\Re} \\ x_j^{\Im}
	\end{bmatrix}
	\\= \begin{bmatrix}
	 \frac{\displaystyle \gamma}{\displaystyle 1+s\gamma g(s)}U U^\top \Omega F(s) e_j \\
	- (\frac{\displaystyle s\gamma^2 f(s)}{\displaystyle 1+s\gamma g(s)}U U^\top \Omega F(s)  + \gamma U U^\top \Omega G(s) ) e_j
	\end{bmatrix},\\
\end{multline*}
\begin{multline*}
	\begin{bmatrix}
	I & -\frac{\displaystyle s\gamma}{\displaystyle 1+s\gamma g(s)}U U^\top F(s) \\
	0 & s^2 \gamma^2 U U^\top F(s)(I-\frac{\displaystyle s\gamma}{\displaystyle 1+s\gamma g(s)}U U^\top G(s))U U^\top F(s) + I +s\gamma U U^\top G(s)
	\end{bmatrix}
	\begin{bmatrix}
	x_j^{\Re} \\ x_j^{\Im}
	\end{bmatrix}
	\\= \begin{bmatrix}
	(I - \frac{\displaystyle s\gamma}{\displaystyle 1+s\gamma g(s)}U U^\top G(s)) e_{j-n} \\
	-(s \gamma U U^\top F(s) + s^2\gamma^2 f(s) U U^\top G(s) ) e_{j-n}
	\end{bmatrix}.
\end{multline*}
Note that since we know that \eqref{eq:sustav1} and \eqref{eq:sustav2} have unique solutions, it follows that the $(2,2)$ entry in the matrix on the left hand side is non-singular.
As the $(2,2)$ entry in the matrix on the left hand side is

\begin{align*}
& s^2 \gamma^2 U U^\top F(s) \left(I-\frac{\displaystyle s\gamma}{\displaystyle 1+s\gamma g(s)}U U^\top G(s)\right)U U^\top F(s) + I +s\gamma U U^\top G(s)  \\
&= I + s \gamma U \left(U^\top G(s) + s \gamma U^\top F(s) \left(I-\frac{s \gamma}{1 + s \gamma g(s)}U U^\top G(s)\right) U U^\top F(s) \right)\\
&= I + s \gamma U \left(U^\top G(s) + \frac{s \gamma f(s)}{1 + s \gamma g(s)} U^\top F(s) \right) ,
\end{align*}
we have

\begin{align*}
&\left(I + s \gamma U \left(U^\top G(s) + \frac{s \gamma f(s)}{1 + s \gamma g(s)} U^\top  F(s)\right) \right)^{-1} \\
&= I- s U \left(\frac{1}{\gamma} + s \left(U^\top  G(s) + \frac{s\gamma f(s)}{1 + s \gamma g(s)} U^\top  F(s)\right)U \right)^{-1} \times\\
& \quad \times \left(U^\top  G(s) + \frac{s\gamma f(s)}{1 + s \gamma g(s)} U^\top  F(s)\right)   \\
&= I - s U \left(\frac{1}{\gamma} + s g(s) + \frac{s^2 \gamma f(s)^2}{1 + s \gamma g(s)}\right)^{-1} \left(U^\top  G(s) + \frac{s \gamma f(s)}{1 + s \gamma g(s)} U^\top  F(s)\right)\\
&= I - \frac{s \gamma}{(1 + s \gamma g(s))^2+(s \gamma f(s))^2} \left((1 + s \gamma g(s)) U U^\top  G(s)
+ s \gamma f(s)U U^\top  F(s) \right).
\end{align*}
Now, the formulae \eqref{xj1}-\eqref{xr2} follow by direct calculation.
\end{proof}

\begin{proof}[Proof of Proposition \ref{prop:solutionHjk}]
First, for $ 1\le k \le n$  we obtain
\begin{displaymath}
(\mathrm{i}s - A_0^\top)^{-1} e_k= \begin{bmatrix}
(\nu \omega_k + \mathrm{i}s)l_k(s) e_k \\ \omega_k l_k(s) e_k
\end{bmatrix},
\end{displaymath}
and for $n+1\le k\le 2n $ we obtain
\begin{displaymath}
(\mathrm{i}s - A_0^\top)^{-1} e_k= \begin{bmatrix}
- \omega_{k-n} l_{k-n}(s) e_{k-n} \\ \mathrm{i}s  l_{k-n}(s) e_{k-n}
\end{bmatrix},
\end{displaymath}
where $l_k(s)=f_k(s) - \mathrm{i}g_k(s)$, for $k=1,\ldots, n.$\\
For $1\le j \le r$ and $1\le k\le n$ we obtain
\begin{align*}
e_j^\top (\mathrm{i}s - A^\top)^{-1}e_k &= (\nu\omega_k  + \mathrm{i}s) l_k(s) e_j^\top  e_k + \omega_k l_k(s) \left( x_j^2 \right)^\top  e_k \\
&= (\nu\omega_k  + \mathrm{i}s) (f_k(s) - \mathrm{i}g_k(s) ) e_j^\top  e_k + \omega_k (f_k(s) - \mathrm{i}g_k(s) ) \left( x_j^2 \right)^\top  e_k\\
&= \left(\nu \omega_k f_k(s) + s g_k(s)\right)\delta_{j,k} + \omega_k \left( f_k(s) (x_j^{\Re})_k +  g_k(s) (x_j^{\Im})_k\right) \\
&\quad + \mathrm{i} \left( (s f_k(s) - \nu \omega_k g_k(s))\delta_{j,k} + \omega_k \left(f_k(s) (x_j^{\Im})_k -  g_k(s) (x_j^{\Re})_k \right)\right) ,
\end{align*}
for $n+1\le j \le n+r$ and $1\le k\le n$ we obtain
\begin{align*}
e_j^\top (\mathrm{i}s - A^\top)^{-1}e_k &=  \omega_k l_k(s) \left( x_j^2 \right)^\top  e_k =  \omega_k (f_k(s) - \mathrm{i}g_k(s) ) \left( x_j^2 \right)^\top  e_k \\
&=  \omega_k \left( f_k(s) (x_j^{\Re})_k + g_k(s)(x_j^{\Im})_k \right) + \mathrm{i} \omega_k \left( f_k(s) (x_j^{\Im})_k - g_k(s) (x_j^{\Re})_k \right) ,
\end{align*}
for $1\le j \le r$ and $n+1\le k\le 2n$ we obtain
\begin{align*}
& e_j^\top (\mathrm{i}s - A^\top)^{-1}e_k = - \omega_{k-n} l_{k-n}(s) e_j^\top  e_{k-n} + \mathrm{i}s l_{k-n}(s) \left( x_j^2 \right)^\top  e_{k-n}\\
 &= - \omega_{k-n} (f_{k-n}(s) - \mathrm{i}g_{k-n}(s) ) e_j^\top  e_{k-n} + \mathrm{i}s (f_{k-n}(s) - \mathrm{i}g_{k-n}(s) ) \left( x_j^2 \right)^\top  e_{k-n}\\
 &= - \omega_{k-n} f_{k-n}(s) \delta_{j,k-n} + s g_{k-n} (s) (x_j^{\Re})_{k-n} - s f_{k-n} (s) (x_j^{\Im})_{k-n} \\
 &\quad + \mathrm{i} \left(  \omega_{k-n} g_{k-n} (s) \delta_{j,k-n}  + s f_{k-n}(s)(x_j^{\Re})_{k-n} + s g_{k-n}(s) (x_j^{\Im})_{k-n}   \right)   ,
\end{align*}
and for $n+1\le j \le n+r$ and $n+1\le k\le 2n$ we obtain
\begin{align*}
e_j^\top (\mathrm{i}s - A^\top)^{-1}e_k &= \mathrm{i}s l_{k-n}(s) \left( x_j^2 \right)^\top  e_{k-n} = \mathrm{i}s (f_{k-n}(s) - \mathrm{i}g_{k-n}(s) ) \left( x_j^2 \right)^\top  e_{k-n} \\
&=  s \left( g_{k-n}(s) (x_j^{\Re})_{k-n}  - f_{k-n} (s) (x_j^{\Im})_{k-n} \right)\\
& \quad + \mathrm{i}  s \left(  f_{k-n}(s) (x_j^{\Re})_{k-n}  + g_{k-n}(s) (x_j^{\Im})_{k-n}  \right) .
\end{align*}
Now we take into account that the inner integral has a real value (see, e.g., \cite{Cohen07}), so we just need to calculate its real part. Thus,
 the function  $h_{jk}$ defined by \eqref{h_{jk}(t,s)} satisfies
\begin{displaymath} h_{jk}(t,s) = 2 \cos{st} \cdot \Re (e_j^\top (\mathrm{i}s-A^\top)^{-1} e_k).\end{displaymath}
The  result now follows directly by taking real parts in the formulae given above.
\end{proof}
\end{section}

\section*{Acknowledgements}
The authors express their gratitude to Zlatko Drma\v{c} for very detailed comments on an earlier version of the manuscript. The authors thank the anonymous referees for their very careful reading of the manuscript and helpful comments.

\end{document}